\documentclass[11pt]{article}

\usepackage{amsmath,amssymb,amsthm}
\usepackage{mathrsfs}
\usepackage{cite}
\usepackage{geometry}
\newtheorem{theorem}{Theorem}[section]
\newtheorem{proposition}[theorem]{Proposition}
\newtheorem{lemma}[theorem]{Lemma}
\newtheorem{corollary}[theorem]{Corollary}
\newtheorem{remark}[theorem]{Remark}
\newtheorem{definition}[theorem]{Definition}
\numberwithin{equation}{section}

\begin{document}

\title{Complete Spectrum and Sharp Local Stability for the Critical Exponential Biharmonic Choquard Equation in \(\mathbb R^{4}\)}

\author{Wenjing Chen and Shengbing Deng\\
School of Mathematics and Statistics, Southwest University, Chongqing, China\\
Email: wjchen@swu.edu.cn, \ \ \ \ shbdeng@swu.edu.cn\\
}

\date{}

\maketitle

\begin{abstract}
We study the conformally invariant exponential biharmonic Choquard equation in
\(\mathbb R^{4}\). Our principal result is the complete spectral resolution of
the linearized operator at its conformal bubbles. After stereographic
projection, the operator becomes a bounded zeroth-order perturbation of the
Paneitz operator, with a compact Riesz component on
\(L^{2}(\mathbb S^{4})\). We justify the weak conformal transfer, remove every
pole-supported distributional defect, and compute all eigenvalues. The Morse
index is one. The kernel is the five-dimensional conformal space. All higher
modes satisfy a uniform coercivity estimate. The transverse Hessian of the
Adams--Choquard deficit is the same linearized operator. Hence the complete
spectrum gives sharp local stability with respect to the Paneitz distance from
the conformal extremal manifold. The optimal asymptotic constant is
\[
\gamma_{\alpha}
=\frac{160-12\alpha-\alpha^{2}}{40(10-\alpha)},
\]
and the limiting quotient is minimized precisely by the second spherical
harmonics. As a nonlinear preparation, we also prove that every normal
distributional finite-mass solution satisfies the hypotheses of Niu's
classification theorem and is therefore an explicit translation--dilation
bubble.
\end{abstract}



\section{Introduction}

Second-order elliptic equations with exponential nonlinearity are the local
models for conformal curvature problems in dimension two. The equation
\(-\Delta u=V(x)e^{u}\) exhibits concentration, quantization and loss of
compactness at the Trudinger--Moser threshold. Brezis and Merle established the
basic local estimates \cite{BrezisMerle}. Chen and Li classified the entire
finite-mass profiles \cite{ChenLi}; quantization and sharp blow-up analysis were
developed further in \cite{LiShafrir,AdimurthiDruet}. Recent variational and
geometric work addresses critical points, compactness and high-energy bubbling
for Moser--Trudinger functionals; see
\cite{DeMarchisMalchiodiMartinazziThizy,MalchiodiMartinazziThizy}.

The fourth-order local endpoint occurs in dimension four. Adams' inequality
\cite{Adams} replaces the Trudinger--Moser inequality, and the Paneitz operator
and \(Q\)-curvature equation replace their second-order counterparts. Entire
finite-volume solutions of conformally invariant polyharmonic equations were
classified in \cite{Lin,WeiXu,Martinazzi}. Concentration--compactness,
compactness and bubbling for fourth-order exponential equations were studied
in \cite{MartinazziCC,DruetRobert,LinWeiFourthOrder,MalchiodiCompactness}.
The geometric variational theory includes
\cite{ChangYang,DjadliMalchiodi,MalchiodiStruwe}.

The Choquard interaction introduces a nonlocal exponential endpoint. Its
variational structure originates in Lieb's sharp
Hardy--Littlewood--Sobolev theory \cite{LiebChoquard,LiebSharp}; the power-type
theory now includes critical ground states, uniqueness, nondegeneracy and
perturbative constructions, as surveyed in \cite{MorozVanSchaftingen} and
developed in
\cite{MorozVanSchaftingenCrit,CassaniZhang,CassaniVanSchaftingenZhang,VanSchaftingenXia,MaZhao,DuYang,GaoMorozYangZhao,SquassinaYangZhao,LiLiuTangXu}.
Detailed eigenvalue asymptotics for slightly subcritical Hartree bubbles and
their consequences for the Morse index were obtained in
\cite{CannoneCingolaniYangZhao}. For the Laplacian in \(\mathbb R^{2}\), nonlinear Choquard equations with
Trudinger--Moser critical growth have been treated by variational methods under
positive, indefinite and decaying potentials; representative results include
\cite{QinTang,CarvalhoMedeirosRibeiro}. The exact conformally invariant model is
\[
-\Delta u
=\left(\int_{\mathbb R^{2}}\frac{e^{u(y)}}{|x-y|^{\alpha}}\,dy\right)e^{u(x)}.
\]
The planar classification was established in \cite{GluckPlanar} and is also
contained in Niu's higher-order classification theorem
\cite{NiuClassification}. Its bubble nondegeneracy was proved recently in
\cite{GaoLiMa}. Concentration--compactness and quantization for related
nonlocal exponential equations were obtained in \cite{Gluck}, and a broader
classification of conformally invariant nonlocal exponential equations appears
in \cite{GluckMarasinghe}. That work classifies the asymptotic behaviour under
a suitable integrability condition and obtains explicit profiles under
additional growth and energy assumptions. Its scope overlaps with the
classification side of the present equation. Our starting point is different:
we assume the normal distributional representation \eqref{eq:normality} and the
single conformal mass condition \eqref{eq:finite-mass}. We then prove
regularity, the exact mass, and the far-field estimates needed to verify the
hypotheses of Niu's theorem. Thus Section~3 is a self-contained reduction and
normalization step. The complete spectrum and the sharp stability result are
the new parts of the paper.

For biharmonic Choquard equations in \(\mathbb R^{4}\), the variational theory
with Adams-critical growth has been developed in
\cite{ChenWangZAMP,ChenLiWang,ChenWangPEMS}. Those works concern existence and
normalized states for nonlinearities with critical exponential growth. The
exact conformal limit equation serves a different purpose. It determines the
bubble profile arising in blow-up and reduction arguments, while its
linearized spectrum governs modulation, quantitative stability, local
uniqueness and finite-dimensional reduction.

In the present paper we consider
\begin{equation}\label{eq:main}
\Delta^{2}u=
\left(\int_{\mathbb R^{4}}\frac{e^{u(y)}}{|x-y|^{\alpha}}\,dy\right)e^{u(x)}
\qquad\hbox{in }\mathbb R^{4},
\end{equation}
where \(0<\alpha<4\). Throughout the paper, all functions are real-valued. This equation combines the logarithmic conformal structure of a four-dimensional Liouville equation with a critical Riesz interaction. It is invariant under translations and dilations: if \(u\) solves \eqref{eq:main}, then
\begin{equation}\label{eq:invariance}
u_{\delta,\zeta}(x)
:=u(\delta(x-\zeta))+\frac{8-\alpha}{2}\log\delta,
\qquad \delta>0,\quad \zeta\in\mathbb R^{4},
\end{equation}
solves the same equation. The corresponding conformally invariant mass is
\begin{equation}\label{eq:finite-mass}
\int_{\mathbb R^{4}}e^{\frac{8}{8-\alpha}u(x)}\,dx<+\infty.
\end{equation}

For the classification step, set
\[
b_{\alpha}:=\frac{8-\alpha}{2},
\qquad
v:=\frac{1}{b_{\alpha}}
\left(u-\frac12\log b_{\alpha}\right).
\]
Then \eqref{eq:main} becomes
\[
\Delta^{2}v=
\left(\int_{\mathbb R^{4}}
\frac{e^{b_{\alpha}v(y)}}{|x-y|^{\alpha}}\,dy\right)
e^{b_{\alpha}v(x)}.
\]
Thus it is the four-dimensional member of the higher-order critical Choquard family. We use the published classification theorem of Niu \cite{NiuClassification}, whose development began in the earlier Huang--Niu preprint \cite{HuangNiu}. The published theorem is formulated for classical solutions satisfying a
subquadratic growth assumption and two global integrability conditions. We
verify these assumptions from the normal distributional hypothesis naturally
arising from the biharmonic logarithmic representation.

For a solution of \eqref{eq:main}, put
\[
F(x):=
\left(\int_{\mathbb R^{4}}\frac{e^{u(y)}}{|x-y|^{\alpha}}\,dy\right)e^{u(x)}.
\]
We call a finite-mass solution normal if
\begin{equation}\label{eq:normality}
u(x)=-\frac{1}{8\pi^{2}}
\int_{\mathbb R^{4}}\log|x-y|F(y)\,dy+C
\end{equation}
for some \(C\in\mathbb R\), where the integral is finite for every
\(x\in\mathbb R^{4}\). The assumption excludes the nonconstant entire
biharmonic remainder, a phenomenon specific to higher-order Liouville
equations. Define
\begin{equation}\label{eq:C-alpha}
C_{\alpha}:=
\left(
\frac{12(8-\alpha)(6-\alpha)(4-\alpha)}{\pi^{2}}
\right)^{\frac{1}{8-\alpha}},
\end{equation}
and, for \(\mu>0\) and \(\zeta\in\mathbb R^{4}\),
\begin{equation*}
U_{\mu,\zeta}(x)
:=\frac{8-\alpha}{2}\log
\left(
\frac{C_{\alpha}\mu}{1+\mu^{2}|x-\zeta|^{2}}
\right).
\end{equation*}
A direct computation gives
\begin{equation*}
\Delta^{2}U_{\mu,\zeta}(x)
=
\frac{48(8-\alpha)\mu^{4}}
{\left(1+\mu^{2}|x-\zeta|^{2}\right)^{4}},
\end{equation*}
and the conformal Riesz identity gives
\begin{equation}\label{eq:Riesz-bubble}
\int_{\mathbb R^{4}}
\frac{e^{U_{\mu,\zeta}(y)}}{|x-y|^{\alpha}}\,dy
=
\frac{48(8-\alpha)}{C_{\alpha}^{4}}
e^{\frac{\alpha}{8-\alpha}U_{\mu,\zeta}(x)}.
\end{equation}

The classification result used below is stated first.

\begin{theorem}\label{thm:classification}
Let \(0<\alpha<4\). Assume that \(u\in L^{1}_{\operatorname{loc}}(\mathbb R^{4})\) is a normal distributional solution of \eqref{eq:main} satisfying \eqref{eq:finite-mass}. Then \(u\in C^{\infty}(\mathbb R^{4})\), and there exist \(\mu>0\) and \(\zeta\in\mathbb R^{4}\) such that
\[
u(x)=U_{\mu,\zeta}(x).
\]
Moreover,
\begin{equation}\label{eq:mass-identity}
\int_{\mathbb R^{4}}e^{\frac{8}{8-\alpha}U_{\mu,\zeta}(x)}\,dx
=\frac{\pi^{2}}{6}C_{\alpha}^{4},
\end{equation}
and
\begin{equation}\label{eq:double-identity}
\int_{\mathbb R^{4}}\int_{\mathbb R^{4}}
\frac{e^{U_{\mu,\zeta}(x)}e^{U_{\mu,\zeta}(y)}}{|x-y|^{\alpha}}\,dxdy
=8\pi^{2}(8-\alpha).
\end{equation}
\end{theorem}

The new part of the paper concerns the full linearized spectrum. Let
\begin{equation}\label{eq:normalized-bubble}
U(x):=U_{1,0}(x)
=\frac{8-\alpha}{2}\log\frac{C_{\alpha}}{1+|x|^{2}}.
\end{equation}
The Euclidean linearized equation is
\begin{equation}\label{eq:linearized}
\begin{aligned}
\Delta^{2}\phi(x)
={}&
\left(\int_{\mathbb R^{4}}
\frac{e^{U(y)}\phi(y)}{|x-y|^{\alpha}}\,dy\right)e^{U(x)}+
\left(\int_{\mathbb R^{4}}
\frac{e^{U(y)}}{|x-y|^{\alpha}}\,dy\right)e^{U(x)}\phi(x).
\end{aligned}
\end{equation}
We use the conformal weighted space
\begin{equation*}
L^{2}_{w}(\mathbb R^{4})
:=
\left\{\phi:
\int_{\mathbb R^{4}}
\frac{|\phi(x)|^{2}}{(1+|x|^{2})^{4}}\,dx<+\infty
\right\}.
\end{equation*}
The five symmetry modes are
\begin{equation}\label{eq:kernel-functions}
\phi_{j}(x)=\frac{(8-\alpha)x_{j}}{1+|x|^{2}},
\qquad 1\leq j\leq4,
\end{equation}
and
\begin{equation}\label{eq:kernel-scale}
\phi_{5}(x)=\frac{8-\alpha}{2}
\frac{1-|x|^{2}}{1+|x|^{2}}.
\end{equation}
They are obtained by differentiating \(U_{\mu,\zeta}\) with respect to translations and dilation.

Let \(S:\mathbb R^{4}\to\mathbb S^{4}\setminus\{p\}\), \(p=(0,0,0,0,-1)\), be stereographic projection and let
\[
P_{4}^{\mathbb S^{4}}
=(-\Delta_{\mathbb S^{4}})(-\Delta_{\mathbb S^{4}}+2).
\]
Define
\[
\mathcal R_{\alpha}\Phi(\xi)
:=\int_{\mathbb S^{4}}\frac{\Phi(\eta)}{|\xi-\eta|^{\alpha}}\,d\eta,
\qquad
\kappa_{\alpha}:=C_{\alpha}^{8-\alpha}2^{-(8-\alpha)},
\]
and
\begin{equation*}
\mathcal L_{\alpha}
:=P_{4}^{\mathbb S^{4}}
-\kappa_{\alpha}
\bigl(\mathcal R_{\alpha}+\mu_{0}(\alpha)I\bigr).
\end{equation*}
Here \(\mu_{k}(\alpha)\) is the Funk--Hecke eigenvalue of \(\mathcal R_{\alpha}\) on the spherical harmonics \(\mathcal H_{k}\).

The spherical operator admits the following complete spectral resolution.

\begin{theorem}\label{thm:full-spectrum}
The operator \(\mathcal L_{\alpha}\), with domain \(H^{4}(\mathbb S^{4})\) in \(L^{2}(\mathbb S^{4})\), is self-adjoint with compact resolvent. Every \(\mathcal H_{k}\) is an eigenspace and the corresponding eigenvalue is
\begin{equation}\label{eq:Lambda-k-intro}
\Lambda_{k}(\alpha)
=k(k+1)(k+2)(k+3)
-\kappa_{\alpha}
\bigl(\mu_{k}(\alpha)+\mu_{0}(\alpha)\bigr).
\end{equation}
More precisely,
\begin{equation*}
\Lambda_{0}(\alpha)=-6(8-\alpha),
\qquad
\Lambda_{1}(\alpha)=0,
\end{equation*}
and, for every \(k\geq2\),
\begin{equation*}
\Lambda_{k}(\alpha)
\geq\frac45 k(k+1)(k+2)(k+3)>0.
\end{equation*}
Consequently, the Morse index of \(\mathcal L_{\alpha}\) is one and
\[
\ker\mathcal L_{\alpha}=\mathcal H_{1}
=\operatorname{span}\{\xi_{1},\ldots,\xi_{5}\}.
\]
Moreover, for every \(\Phi\in H^{2}(\mathbb S^{4})\) satisfying
\(\Phi\perp_{L^{2}}\mathcal H_{0}\oplus\mathcal H_{1}\),
\begin{equation}\label{eq:coercivity-intro}
\langle\mathcal L_{\alpha}\Phi,\Phi\rangle_{L^{2}(\mathbb S^{4})}
\geq
\frac45
\langle P_{4}^{\mathbb S^{4}}\Phi,\Phi\rangle_{L^{2}(\mathbb S^{4})}.
\end{equation}
On the \(L^{2}\)-orthogonal complement of \(\mathcal H_{1}\), the operator is invertible and
\begin{equation*}
\|\Phi\|_{H^{4}(\mathbb S^{4})}
\leq C\|\mathcal L_{\alpha}\Phi\|_{L^{2}(\mathbb S^{4})},
\qquad \Phi\perp_{L^{2}}\mathcal H_{1},
\end{equation*}
where \(C\) is independent of \(0<\alpha<4\).
\end{theorem}

The factor \(4/5\) is uniform for \(0<\alpha<4\). We do not claim that it is
the optimal coercivity constant for a fixed value of \(\alpha\). The sharp
quantity used in the stability theorem is instead the mode quotient
\(\gamma_{\alpha}\) in \eqref{eq:sharp-local-constant}.

The passage from \(\mathbb R^{4}\) to the compact sphere has a subtle point. A weighted \(L^{2}\)-function becomes an \(L^{2}\)-function on \(\mathbb S^{4}\), but a priori the transformed equation is known only away from the stereographic pole. Since the Green function of a fourth-order operator is logarithmic in dimension four, an \(L^{2}\)-function can still carry a point-mass defect. We prove that the only possible defect consists of a Dirac mass and its first tangential derivatives. Pairing the defect with the five first spherical harmonics, which already lie in the kernel, eliminates all coefficients. This gives a global sphere equation without first proving Euclidean boundedness or using a Kelvin transform across the origin.

The removal of the pole defect yields the corresponding Euclidean kernel statement.

\begin{corollary}\label{cor:euclidean-nondegeneracy}
Let \(\phi\in L^{2}_{w}(\mathbb R^{4})\) be a distributional solution of \eqref{eq:linearized}. Then \(\phi\in C^{\infty}(\mathbb R^{4})\cap L^{\infty}(\mathbb R^{4})\) and
\[
\phi=\sum_{j=1}^{5}a_{j}\phi_{j}
\]
for suitable constants \(a_{1},\ldots,a_{5}\). The same conclusion holds at
every bubble \(U_{\mu,\zeta}\) after translation and dilation. In addition, on
the Euclidean form domain corresponding to \(H^{2}(\mathbb S^{4})\) under
stereographic projection, the quadratic form has Morse index one. Its kernel
is generated by the five symmetry modes, and it is coercive on the orthogonal
complement of the constant mode and those symmetry modes.
\end{corollary}

The complete spectrum also yields a quantitative result for a conformally invariant functional. Let \(|\mathbb S^{4}|=8\pi^{2}/3\), and write
\[
\overline{\Phi}:=\frac{1}{|\mathbb S^{4}|}
\int_{\mathbb S^{4}}\Phi\,d\xi.
\]
Set
\begin{equation*}
\mathcal I_{\alpha}(\Phi)
:=\int_{\mathbb S^{4}}\int_{\mathbb S^{4}}
\frac{e^{\Phi(\xi)}e^{\Phi(\eta)}}{|\xi-\eta|^{\alpha}}\,d\xi d\eta
\end{equation*}
and
\begin{equation}\label{eq:deficit}
\begin{aligned}
\mathcal D_{\alpha}(\Phi)
:={}&\frac12\int_{\mathbb S^{4}}
\Phi P_{4}^{\mathbb S^{4}}\Phi\,d\xi
+8\pi^{2}(8-\alpha)\overline{\Phi}-4\pi^{2}(8-\alpha)
\log\frac{\mathcal I_{\alpha}(\Phi)}
{|\mathbb S^{4}|\mu_{0}(\alpha)}.
\end{aligned}
\end{equation}
For \(a\in B^{5}:=\{a\in\mathbb R^{5}:|a|<1\}\), define
\begin{equation}\label{eq:J-a}
J_{a}(\xi)
:=\left(\frac{1-|a|^{2}}
{1-2a\cdot\xi+|a|^{2}}\right)^{4}
\end{equation}
and
\begin{equation*}
\mathcal M_{\alpha}
:=\left\{c+\frac{8-\alpha}{8}\log J_{a}:
 c\in\mathbb R,\ a\in B^{5}\right\}.
\end{equation*}
Since \(P_{4}^{\mathbb S^{4}}\) annihilates constants, it is natural to use
\begin{equation*}
\dot H^{2}(\mathbb S^{4}):=H^{2}(\mathbb S^{4})/\mathbb R,
\qquad
\|[\Phi]\|_{P}^{2}:=
\int_{\mathbb S^{4}}\Phi P_{4}^{\mathbb S^{4}}\Phi\,d\xi.
\end{equation*}
The Paneitz distance to the extremal manifold is
\begin{equation*}
d_{P}(\Phi,\mathcal M_{\alpha})^{2}
:=\inf_{\Psi\in\mathcal M_{\alpha}}
\int_{\mathbb S^{4}}(\Phi-\Psi)
P_{4}^{\mathbb S^{4}}(\Phi-\Psi)\,d\xi.
\end{equation*}

The sharp endpoint inequality takes the following form.

\begin{corollary}\label{cor:sharp-Adams-Choquard}
Let \(0<\alpha<4\). For every \(\Phi\in H^{2}(\mathbb S^{4})\),
\begin{equation}\label{eq:sharp-Adams-Choquard}
\log\frac{\mathcal I_{\alpha}(\Phi)}
{|\mathbb S^{4}|\mu_{0}(\alpha)}
\leq2\overline{\Phi}
+\frac{1}{8\pi^{2}(8-\alpha)}
\int_{\mathbb S^{4}}\Phi P_{4}^{\mathbb S^{4}}\Phi\,d\xi.
\end{equation}
Equality holds exactly for \(\Phi\in\mathcal M_{\alpha}\).
\end{corollary}

The complete spectral information also determines the sharp local remainder.

\begin{theorem}\label{thm:Adams-Choquard-stability}
Let \(0<\alpha<4\). There exist \(\varepsilon_{\alpha}>0\) and
\(c_{\alpha}>0\) such that
\begin{equation}\label{eq:local-stability-intro}
\mathcal D_{\alpha}(\Phi)
\geq c_{\alpha}d_{P}(\Phi,\mathcal M_{\alpha})^{2}
\end{equation}
whenever \(d_{P}(\Phi,\mathcal M_{\alpha})<\varepsilon_{\alpha}\). Moreover,
\begin{equation}\label{eq:sharp-local-constant}
\lim_{\varepsilon\downarrow0}
\inf_{0<d_{P}(\Phi,\mathcal M_{\alpha})<\varepsilon}
\frac{\mathcal D_{\alpha}(\Phi)}
{d_{P}(\Phi,\mathcal M_{\alpha})^{2}}
=\gamma_{\alpha}
:=\frac{\Lambda_{2}(\alpha)}{240}
=\frac{160-12\alpha-\alpha^{2}}{40(10-\alpha)}
>\frac25.
\end{equation}
The constants in \eqref{eq:local-stability-intro} can be chosen uniformly when
\(\alpha\) ranges in a compact subset of \((0,4)\).
\end{theorem}

The distinction between the last two statements is important. Corollary
\ref{cor:sharp-Adams-Choquard} is obtained by composing the sharp spherical
Hardy--Littlewood--Sobolev inequality with Beckner's four-dimensional Adams
inequality. The new assertion is Theorem~\ref{thm:Adams-Choquard-stability}.
The Hessian of \(\mathcal D_\alpha\), after the constant and conformal directions
are removed, is exactly \(\mathcal L_\alpha\). The full spectrum therefore gives
not only positivity but the sharp transverse constant. The second spherical
harmonic space is the unique lowest positive mode.

This mechanism differs from the stability of either factor inequality alone.
Endpoint stability for log-Sobolev and Moser--Onofri inequalities was studied in
\cite{ChenLuTang}; quantitative HLS stability with explicit lower bounds was
developed in \cite{ChenLuTangHLS}. Recent work on Trudinger--Moser stability
appears in \cite{AndradeOliveiraDoMacedoRatzkin}. In the present problem, the
Hessian contains the coupled operator
\(P_{4}^{\mathbb S^{4}}-\kappa_\alpha(\mathcal R_\alpha+\mu_0I)\), and the sharp
constant depends explicitly on the Riesz parameter \(\alpha\). We prove a local
statement. No global Bianchi--Egnell-type estimate in the sense of \cite{BianchiEgnell} is asserted.

The functional-analytic core of the paper is the self-adjoint operator
\(\mathcal L_{\alpha}\), a bounded zeroth-order perturbation of
\(P_{4}^{\mathbb S^{4}}\) with compact Riesz component. Its simultaneous
diagonalization with the spherical Riesz operator produces the complete
spectrum, rather than kernel information alone. The contributions are as
follows. First, Section~3 gives a self-contained reduction and normalization
from the normal distributional finite-mass class to the published
classification theorem. Second, we establish the weak conformal transfer,
eliminate every pole-supported defect, and compute the negative eigenspace,
the conformal kernel, and all positive modes with a uniform gap. Third, we
derive a nearest-point modulation theorem and the sharp local stability
constant in \eqref{eq:sharp-local-constant}. Compared with the planar
exponential Choquard nondegeneracy result \cite{GaoLiMa}, the fourth-order
problem has a genuine Paneitz spectrum and a critical logarithmic pole.
Compared with power-type critical Hartree bubbles
\cite{DuYang,LiLiuTangXu}, the exponential model has zero conformal weight on
functions in dimension four. This makes the complete spherical
diagonalization possible.

The paper is organized as follows. Section 2 records the sharp inequalities,
conformal identities and Funk--Hecke eigenvalues. Section 3 proves the reduction
from normal distributional solutions to the known classification theorem.
Section 4 constructs the global spherical linearized operator and removes the
pole defect. Section 5 computes the complete spectrum. Section 6 proves the
sharp Adams--Choquard inequality, the nearest-point modulation result and the
sharp local stability theorem, together with their Euclidean formulation.

\section{Preliminaries}

We first recall the Hardy--Littlewood--Sobolev inequality.

\begin{lemma}\label{lem:HLS}
Let \(N\geq1\), \(0<\alpha<N\), and \(p,r>1\) satisfy
\[
\frac1p+\frac1r+\frac{\alpha}{N}=2.
\]
Then there exists \(C>0\) such that
\[
\int_{\mathbb R^{N}}\int_{\mathbb R^{N}}
\frac{f(x)g(y)}{|x-y|^{\alpha}}\,dxdy
\leq C\|f\|_{L^{p}(\mathbb R^{N})}\|g\|_{L^{r}(\mathbb R^{N})}.
\]
When \(p=r=\frac{2N}{2N-\alpha}\), the sharp constant is
\[
C(N,\alpha)=
\pi^{\frac{\alpha}{2}}
\frac{\Gamma\left(\frac{N-\alpha}{2}\right)}
{\Gamma\left(N-\frac{\alpha}{2}\right)}
\left(
\frac{\Gamma(N)}{\Gamma\left(\frac N2\right)}
\right)^{\frac{N-\alpha}{N}}.
\]
\end{lemma}

For \(N=4\) and \(p=r=\frac{8}{8-\alpha}\),
\begin{equation*}
C(4,\alpha)=
\pi^{\frac{\alpha}{2}}
\frac{\Gamma\left(2-\frac{\alpha}{2}\right)}
{\Gamma\left(4-\frac{\alpha}{2}\right)}
6^{\frac{4-\alpha}{4}}.
\end{equation*}

\begin{lemma}\label{lem:Riesz-estimate}
Let \(0<\alpha<4\), \(\theta>0\), and \(\theta+\alpha>4\). Then
\[
\int_{\mathbb R^{4}}
\frac{1}{|x-y|^{\alpha}}
\frac{1}{(1+|y|^{2})^{\theta/2}}\,dy
\leq C
\begin{cases}
(1+|x|^{2})^{(4-\alpha-\theta)/2},&\theta<4,\\
(1+|x|^{2})^{-\alpha/2}\log(2+|x|),&\theta=4,\\
(1+|x|^{2})^{-\alpha/2},&\theta>4.
\end{cases}
\]
\end{lemma}

\begin{proof}
The estimate is immediate for \(|x|\leq2\), because
\(\alpha<4\) and \(\theta+\alpha>4\). Assume that \(|x|>2\), and
split \(\mathbb R^{4}\) into
\[
E_{1}=\{|y|\leq |x|/2\},\qquad
E_{2}=\{|x|/2<|y|<2|x|\},\qquad
E_{3}=\{|y|\geq2|x|\}.
\]
On \(E_{1}\), \(|x-y|\geq |x|/2\), and hence
\[
\int_{E_{1}}\frac{dy}{|x-y|^{\alpha}(1+|y|^{2})^{\theta/2}}
\leq C|x|^{-\alpha}
\int_{0}^{|x|/2}\frac{r^{3}}{(1+r^{2})^{\theta/2}}\,dr.
\]
The last integral is bounded by \(C|x|^{4-\theta}\) if
\(\theta<4\), by \(C\log|x|\) if \(\theta=4\), and by a constant
if \(\theta>4\). On \(E_{2}\),
\((1+|y|^{2})^{-\theta/2}\leq C|x|^{-\theta}\), and therefore
\[
\int_{E_{2}}\frac{dy}{|x-y|^{\alpha}(1+|y|^{2})^{\theta/2}}
\leq C|x|^{-\theta}
\int_{B_{3|x|}(x)}|x-y|^{-\alpha}\,dy
\leq C|x|^{4-\alpha-\theta}.
\]
Finally, on \(E_{3}\), \(|x-y|\geq |y|/2\), so
\[
\int_{E_{3}}\frac{dy}{|x-y|^{\alpha}(1+|y|^{2})^{\theta/2}}
\leq C\int_{2|x|}^{\infty}r^{3-\alpha-\theta}\,dr
\leq C|x|^{4-\alpha-\theta}.
\]
Combining the three estimates and replacing powers of \(|x|\) by the
equivalent powers of \(1+|x|^{2}\) proves the result.
\end{proof}

\begin{lemma}\label{lem:Riesz-regularity}
Let \(0<\alpha<4\), \(k\geq0\), and \(0<\beta<1\). If \(g\in C^{k,\beta}_{c}(\mathbb R^{4})\), then
\[
I_{\alpha}g(x):=
\int_{\mathbb R^{4}}\frac{g(y)}{|x-y|^{\alpha}}\,dy
\]
belongs to \(C^{k,\beta}_{\operatorname{loc}}(\mathbb R^{4})\).
\end{lemma}

\begin{proof}
Let \(K\Subset\mathbb R^{4}\), and choose a ball \(B\) containing both
\(K\) and the support of \(g\). Write
\(K_{\alpha}(x)=|x|^{-\alpha}\). For every multi-index
\(\gamma\) with \(|\gamma|\leq k\), distributional convolution gives
\[
D^{\gamma}(K_{\alpha}*g)
=K_{\alpha}*D^{\gamma}g.
\]
This identity follows by moving derivatives from the kernel to the compactly
supported density. Since \(K_{\alpha}\in L^{1}_{\operatorname{loc}}
(\mathbb R^{4})\), the right-hand side is continuous on \(K\). If
\(|\gamma|=k\), then, for \(x,x+h\in K\),
\[
\begin{aligned}
&|K_{\alpha}*D^{\gamma}g(x+h)
-K_{\alpha}*D^{\gamma}g(x)|\\
&\quad=\left|\int_{\mathbb R^{4}}K_{\alpha}(x-y)
\bigl(D^{\gamma}g(y+h)-D^{\gamma}g(y)\bigr)\,dy\right|
\leq C_{K,g}|h|^{\beta}.
\end{aligned}
\]
Here we used a change of variables in the first term and enlarged the fixed
ball containing the supports of the translated densities. The integral of
\(K_{\alpha}\) over that ball is finite. Thus all derivatives of order
\(k\) are locally \(\beta\)-H\"older continuous. The lower derivatives
are treated in the same way. This proves
\(I_{\alpha}g\in C^{k,\beta}_{\operatorname{loc}}(\mathbb R^{4})\);
see also \cite{Stein}.
\end{proof}

The logarithmic fundamental solution of \(\Delta^{2}\) in \(\mathbb R^{4}\) satisfies
\begin{equation*}
\Delta^{2}\left(-\frac{1}{8\pi^{2}}\log|x|\right)=\delta_{0}
\qquad\hbox{in }\mathcal D'(\mathbb R^{4}).
\end{equation*}

Let \(S:\mathbb R^{4}\to\mathbb S^{4}\setminus\{p\}\), \(p=(0,0,0,0,-1)\), be
\begin{equation}\label{eq:stereo}
S(x)=
\left(
\frac{2x}{1+|x|^{2}},
\frac{1-|x|^{2}}{1+|x|^{2}}
\right),
\end{equation}
and set
\begin{equation}\label{eq:omega}
\omega(x):=\frac{2}{1+|x|^{2}}.
\end{equation}
Then
\begin{equation}\label{eq:stereo-identities}
d\xi=\omega(x)^{4}\,dx,
\qquad
|S(x)-S(y)|=\omega(x)^{1/2}\omega(y)^{1/2}|x-y|.
\end{equation}
For a function \(f\) on \(\mathbb R^{4}\), write \(S_{*}f=f\circ S^{-1}\), and for a function \(\Phi\) on \(\mathbb S^{4}\), write \(S^{*}\Phi=\Phi\circ S\). Thus
\begin{equation}\label{eq:L2-isometry}
\int_{\mathbb S^{4}}|S_{*}f|^{2}\,d\xi
=16\int_{\mathbb R^{4}}
\frac{|f(x)|^{2}}{(1+|x|^{2})^{4}}\,dx.
\end{equation}

The Paneitz operator on the round four-sphere is
\begin{equation*}
P_{4}^{\mathbb S^{4}}
=(-\Delta_{\mathbb S^{4}})(-\Delta_{\mathbb S^{4}}+2).
\end{equation*}
Its conformal covariance in dimension four gives
\begin{equation}\label{eq:Paneitz-covariance}
\Delta^{2}(S^{*}\Phi)(x)
=\omega(x)^{4}
S^{*}\bigl(P_{4}^{\mathbb S^{4}}\Phi\bigr)(x)
\end{equation}
for smooth \(\Phi\); see \cite{Paneitz,Beckner}. If \(\mathcal H_{k}\) denotes the spherical harmonics of degree \(k\), then
\[
L^{2}(\mathbb S^{4})=\bigoplus_{k=0}^{\infty}\mathcal H_{k},
\qquad
-\Delta_{\mathbb S^{4}}Y=k(k+3)Y,
\]
and therefore
\begin{equation}\label{eq:Paneitz-eigen}
P_{4}^{\mathbb S^{4}}Y
=k(k+1)(k+2)(k+3)Y,
\qquad Y\in\mathcal H_{k}.
\end{equation}
Moreover,
\[
\mathcal H_{1}=\operatorname{span}\{\xi_{1},\xi_{2},\xi_{3},\xi_{4},\xi_{5}\}.
\]

We next record the spherical diagonalization of the Riesz kernel.

\begin{lemma}\label{lem:Funk-Hecke}
For \(0<\alpha<4\), \(k\geq0\), and \(Y\in\mathcal H_{k}\),
\begin{equation*}
\int_{\mathbb S^{4}}\frac{Y(\eta)}{|\xi-\eta|^{\alpha}}\,d\eta
=\mu_{k}(\alpha)Y(\xi),
\end{equation*}
where
\begin{equation}\label{eq:mu-k}
\mu_{k}(\alpha)=
2^{4-\alpha}\pi^{2}
\frac{
\Gamma\left(2-\frac{\alpha}{2}\right)
\Gamma\left(k+\frac{\alpha}{2}\right)
}{
\Gamma\left(\frac{\alpha}{2}\right)
\Gamma\left(k+4-\frac{\alpha}{2}\right)
}.
\end{equation}
In particular,
\begin{equation}\label{eq:mu0-mu1}
\mu_{0}(\alpha)=
\frac{2^{6-\alpha}\pi^{2}}{(6-\alpha)(4-\alpha)},
\qquad
\mu_{1}(\alpha)=
\frac{2^{6-\alpha}\pi^{2}\alpha}
{(8-\alpha)(6-\alpha)(4-\alpha)}.
\end{equation}
Furthermore,
\begin{equation}\label{eq:mu-monotone}
\frac{\mu_{k+1}(\alpha)}{\mu_{k}(\alpha)}
=\frac{k+\frac{\alpha}{2}}{k+4-\frac{\alpha}{2}}<1.
\end{equation}
\end{lemma}

\begin{proof}
Apply the Funk--Hecke formula to the zonal kernel \(|\xi-\eta|^{-\alpha}=(2(1-\xi\cdot\eta))^{-\alpha/2}\). The beta integral gives \eqref{eq:mu-k}; the formulas for \(k=0,1\) follow from \(\Gamma(s+1)=s\Gamma(s)\), and \eqref{eq:mu-monotone} follows by taking the quotient of consecutive eigenvalues; see \cite{AtkinsonHan,DaiXu}.
\end{proof}

The volume of the round sphere is
\begin{equation}\label{eq:sphere-volume}
|\mathbb S^{4}|=\frac{8\pi^{2}}{3}.
\end{equation}

The Euclidean sharp inequality transfers to the sphere as follows.

\begin{lemma}\label{lem:spherical-HLS}
Let
\begin{equation}\label{eq:q-alpha}
q_{\alpha}:=\frac{8}{8-\alpha}.
\end{equation}
For every nonnegative \(f\in L^{q_{\alpha}}(\mathbb S^{4})\),
\begin{equation*}
\frac{1}{|\mathbb S^{4}|\mu_{0}(\alpha)}
\int_{\mathbb S^{4}}\int_{\mathbb S^{4}}
\frac{f(\xi)f(\eta)}{|\xi-\eta|^{\alpha}}\,d\xi d\eta
\leq
\left(\frac{1}{|\mathbb S^{4}|}
\int_{\mathbb S^{4}}f^{q_{\alpha}}\,d\xi\right)^{2/q_{\alpha}}.
\end{equation*}
Equality holds if and only if
\begin{equation*}
f(\xi)=cJ_{a}(\xi)^{1/q_{\alpha}}
\end{equation*}
for some \(c>0\) and \(a\in B^{5}\), where \(J_{a}\) is defined by \eqref{eq:J-a}.
\end{lemma}

\begin{proof}
Apply the sharp Euclidean Hardy--Littlewood--Sobolev inequality in Lemma~\ref{lem:HLS} and use stereographic projection. Formula \eqref{eq:stereo-identities} cancels the conformal weights because
\begin{equation}\label{eq:HLS-weight-cancellation}
\frac{1}{q_{\alpha}}+\frac{\alpha}{8}=1.
\end{equation}
The constant is fixed by testing \(f\equiv1\). Indeed,
\begin{equation*}
\int_{\mathbb S^{4}}\int_{\mathbb S^{4}}
\frac{1}{|\xi-\eta|^{\alpha}}\,d\xi d\eta
=|\mathbb S^{4}|\mu_{0}(\alpha).
\end{equation*}
The equality statement is the spherical form of Lieb's classification of HLS extremals; see \cite{FrankLieb,LiebSharp}.
\end{proof}

The second endpoint ingredient is Beckner's four-dimensional Adams inequality.

\begin{lemma}\label{lem:Beckner-Adams}
For every \(W\in H^{2}(\mathbb S^{4})\),
\begin{equation}\label{eq:Beckner-Adams}
\log\left(\frac{1}{|\mathbb S^{4}|}
\int_{\mathbb S^{4}}e^{4(W-\overline W)}\,d\xi\right)
\leq\frac{1}{8\pi^{2}}
\int_{\mathbb S^{4}}W P_{4}^{\mathbb S^{4}}W\,d\xi.
\end{equation}
Equality holds if and only if
\begin{equation}\label{eq:Beckner-Adams-equality}
W(\xi)=c+\frac14\log J_{a}(\xi)
\end{equation}
for some \(c\in\mathbb R\) and \(a\in B^{5}\).
\end{lemma}

\begin{proof}
This is the four-dimensional endpoint inequality of Beckner \cite{Beckner}. The coefficient can also be checked on the first harmonic modes. The Paneitz eigenvalue on \(\mathcal H_{1}\) is \(24\). The second variation of the logarithm on the left-hand side of \eqref{eq:Beckner-Adams} is \(16|\mathbb S^{4}|^{-1}\|W\|_{2}^{2}\). Formula \eqref{eq:sphere-volume} gives equality of the two second variations on \(\mathcal H_{1}\). Beckner's equality classification gives \eqref{eq:Beckner-Adams-equality}.
\end{proof}

\begin{lemma}\label{lem:convolution-identity}
For \(0<\alpha<4\),
\begin{equation}\label{eq:conv-base}
\int_{\mathbb R^{4}}
\frac{1}{|x-y|^{\alpha}}
\frac{1}{(1+|y|^{2})^{\frac{8-\alpha}{2}}}\,dy
=
\frac{4\pi^{2}}{(4-\alpha)(6-\alpha)}
\frac{1}{(1+|x|^{2})^{\alpha/2}}.
\end{equation}
Consequently \eqref{eq:Riesz-bubble} holds.
\end{lemma}

\begin{proof}
Conformal covariance, or the equality case in the Hardy--Littlewood--Sobolev inequality, shows that the left-hand side of \eqref{eq:conv-base} is \(A_{\alpha}(1+|x|^{2})^{-\alpha/2}\). At \(x=0\),
\[
A_{\alpha}=2\pi^{2}\int_{0}^{\infty}
\frac{r^{3-\alpha}}{(1+r^{2})^{\frac{8-\alpha}{2}}}\,dr
=\pi^{2}B\left(2-\frac{\alpha}{2},2\right)
=\frac{4\pi^{2}}{(4-\alpha)(6-\alpha)}.
\]
Substituting the definition of \(U_{\mu,\zeta}\) and using \eqref{eq:C-alpha} gives \eqref{eq:Riesz-bubble} after scaling and translation.
\end{proof}

\section{Classification of normal finite-mass solutions}

In this section we prove Theorem~\ref{thm:classification}. We first establish regularity, the logarithmic representation, and the exact mass identity in the normal distributional class.  

\begin{definition}
Let \(f\in L^{1}_{\operatorname{loc}}(\mathbb R^{4})\). A function \(u\in L^{1}_{\operatorname{loc}}(\mathbb R^{4})\) is called a distributional solution of
\[
\Delta^{2}u=f
\qquad\hbox{in }\mathbb R^{4}
\]
if
\[
\int_{\mathbb R^{4}}u(x)\Delta^{2}\varphi(x)\,dx
=
\int_{\mathbb R^{4}}f(x)\varphi(x)\,dx
\]
for every \(\varphi\in C_{c}^{\infty}(\mathbb R^{4})\).
\end{definition}

For a solution \(u\) of \eqref{eq:main}, define
\[
F(x):=
\left(\int_{\mathbb R^{4}}\frac{e^{u(y)}}{|x-y|^{\alpha}}\,dy\right)e^{u(x)}.
\]
A solution satisfying \eqref{eq:finite-mass} and \eqref{eq:normality} will be called a normal finite-mass solution.

\begin{lemma}\label{lem:classification-regularity}
Assume that \(u\in L^{1}_{\operatorname{loc}}(\mathbb R^{4})\) solves \eqref{eq:main} and satisfies \eqref{eq:finite-mass}. Then \(F\in L^{1}_{\operatorname{loc}}(\mathbb R^{4})\), and \(u\in C^{\infty}(\mathbb R^{4})\).
\end{lemma}

\begin{proof}
Set
\[
p_{\alpha}:=\frac{8}{8-\alpha},\qquad
f(x):=e^{u(x)}
\qquad
\mbox{and}\quad
V(x):=\int_{\mathbb R^{4}}\frac{f(y)}{|x-y|^{\alpha}}\,dy .
\]
Then the right-hand side of \eqref{eq:main} can be written as
\[
F(x)=V(x)f(x).
\]
By the finite-mass condition \eqref{eq:finite-mass}, we have
\[
f=e^{u}\in L^{p_{\alpha}}(\mathbb R^{4}).
\]
Since
\[
\frac{1}{p_{\alpha}}+\frac{1}{p_{\alpha}}+\frac{\alpha}{4}
=
\frac{8-\alpha}{8}+\frac{8-\alpha}{8}+\frac{\alpha}{4}
=2,
\]
the Hardy--Littlewood--Sobolev inequality gives
\[
\int_{\mathbb R^{4}}\int_{\mathbb R^{4}}
\frac{f(x)f(y)}{|x-y|^{\alpha}}\,dxdy
\leq
C
\|f\|_{L^{p_{\alpha}}(\mathbb R^{4})}^{2}
<+\infty .
\]
By Tonelli's theorem, this implies
\[
\int_{\mathbb R^{4}}F(x)\,dx
=
\int_{\mathbb R^{4}}
\left(
\int_{\mathbb R^{4}}\frac{f(y)}{|x-y|^{\alpha}}\,dy
\right)f(x)\,dx
<+\infty .
\]
Consequently \(F\geq0\) and \(F\in L^{1}(\mathbb R^{4})\), and in
particular \(F\in L^{1}_{\operatorname{loc}}(\mathbb R^{4})\).

It remains to prove the smoothness of \(u\). The preceding argument only gives \(F\in L^{1}\), which is not enough by itself to conclude smoothness from the local \(L^{q}\)-theory. We first prove a local improvement of the integrability of \(e^{u}\).

Fix \(x_{0}\in\mathbb R^{4}\) and \(s>1\). Since \(F\in L^{1}(\mathbb R^{4})\), we can choose \(r>0\) sufficiently small such that
\[
m_{r}:=\int_{B_{2r}(x_{0})}F(y)\,dy
\]
satisfies
\[
\frac{s\,m_{r}}{8\pi^{2}}<4.
\]
Define
\[
v_{r}(x):=
-\frac{1}{8\pi^{2}}
\int_{B_{2r}(x_{0})}\log |x-y|\,F(y)\,dy .
\]
Since
\[
\Delta^{2}\left(-\frac{1}{8\pi^{2}}\log |x|\right)=\delta_{0}
\qquad \hbox{in } \mathcal D'(\mathbb R^{4}),
\]
we have
\[
\Delta^{2}v_{r}=F
\qquad \hbox{in } B_{2r}(x_{0})
\]
in the sense of distributions. Therefore
\[
h_{r}:=u-v_{r}
\]
satisfies
\[
\Delta^{2}h_{r}=0
\qquad \hbox{in } B_{2r}(x_{0}).
\]
Since \(u\in L^{1}_{\operatorname{loc}}(\mathbb R^{4})\) and \(v_{r}\in L^{1}_{\operatorname{loc}}(B_{2r}(x_{0}))\), we have \(h_{r}\in L^{1}_{\operatorname{loc}}(B_{2r}(x_{0}))\). By Weyl's lemma for the biharmonic operator,
\[
h_{r}\in C^{\infty}(B_{2r}(x_{0})).
\]
In particular,
\[
\sup_{B_{r}(x_{0})}|h_{r}|<+\infty .
\]

For \(x\in B_{r}(x_{0})\) and \(y\in B_{2r}(x_{0})\), one has \(|x-y|\leq 3r\). Hence
\[
-\log |x-y|
=
\log\frac{4r}{|x-y|}-\log(4r),
\]
and therefore
\[
v_{r}(x)
\leq
C
+
\frac{1}{8\pi^{2}}
\int_{B_{2r}(x_{0})}
\log\frac{4r}{|x-y|}\,F(y)\,dy ,
\]
where \(C\) depends on \(r\), \(s\), and \(m_{r}\). If \(m_{r}=0\), then \(v_{r}\) is constant in \(B_{2r}(x_{0})\), and the desired local integrability is immediate. Assume now that \(m_{r}>0\). Put
\[
d\nu(y):=\frac{F(y)}{m_{r}}\,dy
\qquad \hbox{on } B_{2r}(x_{0}).
\]
Then \(\nu\) is a probability measure. By Jensen's inequality, for \(x\in B_{r}(x_{0})\),
\[
\begin{aligned}
\exp\left(
\frac{s}{8\pi^{2}}
\int_{B_{2r}(x_{0})}
\log\frac{4r}{|x-y|}\,F(y)\,dy
\right)
&=
\exp\left(
\frac{s\,m_{r}}{8\pi^{2}}
\int_{B_{2r}(x_{0})}
\log\frac{4r}{|x-y|}\,d\nu(y)
\right)  \\
&\leq
\int_{B_{2r}(x_{0})}
\left(\frac{4r}{|x-y|}\right)^{\frac{s\,m_{r}}{8\pi^{2}}}
\,d\nu(y).
\end{aligned}
\]
Since
$
\frac{s\,m_{r}}{8\pi^{2}}<4,
$
we obtain
\[
\begin{aligned}
\int_{B_{r}(x_{0})}e^{s u(x)}\,dx
&\leq
C
\int_{B_{r}(x_{0})}e^{s v_{r}(x)}\,dx  \leq
C
\int_{B_{2r}(x_{0})}
\int_{B_{r}(x_{0})}
|x-y|^{-\frac{s\,m_{r}}{8\pi^{2}}}
\,dx\,d\nu(y)
<+\infty .
\end{aligned}
\]
Thus for every \(x_{0}\in\mathbb R^{4}\) and every \(s>1\), there exists \(r>0\) such that
\[
e^{u}\in L^{s}(B_{r}(x_{0})).
\]
By a finite covering argument, we conclude that
\[
e^{u}\in L^{s}_{\operatorname{loc}}(\mathbb R^{4})
\qquad \hbox{for every } 1<s<+\infty .
\]

We next show that the nonlocal potential \(V\) is locally bounded. Let \(K\Subset\mathbb R^{4}\). Choose \(R>1\) such that
$
K\subset B_{R}(0).
$
For \(x\in K\), decompose
\[
V(x)
=
\int_{B_{2R}(0)}
\frac{e^{u(y)}}{|x-y|^{\alpha}}\,dy
+
\int_{\mathbb R^{4}\setminus B_{2R}(0)}
\frac{e^{u(y)}}{|x-y|^{\alpha}}\,dy
=:V_{1}(x)+V_{2}(x).
\]
Choose \(s>4/(4-\alpha)\). Then its conjugate exponent \(s'=s/(s-1)\) satisfies
$
\alpha s'<4.
$
Using H\"{o}lder's inequality and the local integrability just proved, we get
\[
\begin{aligned}
V_{1}(x)
&\leq
\left(
\int_{B_{2R}(0)}e^{su(y)}\,dy
\right)^{1/s}
\left(
\int_{B_{2R}(0)}|x-y|^{-\alpha s'}\,dy
\right)^{1/s'} \leq C_{K,R}.
\end{aligned}
\]
For \(V_{2}\), since \(x\in B_{R}(0)\) and \(y\in\mathbb R^{4}\setminus B_{2R}(0)\), one has
$
|x-y|\geq \frac{|y|}{2}.
$
Let \(p_{\alpha}'=8/\alpha\) be the conjugate exponent of \(p_{\alpha}=8/(8-\alpha)\). Then
$
\alpha p_{\alpha}'=8>4.
$
Again by H\"{o}lder's inequality,
\[
\begin{aligned}
V_{2}(x)
&\leq
C
\left(
\int_{\mathbb R^{4}}e^{p_{\alpha}u(y)}\,dy
\right)^{1/p_{\alpha}}
\left(
\int_{\mathbb R^{4}\setminus B_{2R}(0)}
|y|^{-\alpha p_{\alpha}'}\,dy
\right)^{1/p_{\alpha}'}
<+\infty .
\end{aligned}
\]
Therefore
$
V\in L^{\infty}_{\operatorname{loc}}(\mathbb R^{4}).
$

Since \(e^{u}\in L^{s}_{\operatorname{loc}}(\mathbb R^{4})\) for every \(s>1\), it follows that
\[
F=Ve^{u}\in L^{s}_{\operatorname{loc}}(\mathbb R^{4})
\qquad \hbox{for every } 1<s<+\infty .
\]

We now apply the local regularity theory for the biharmonic operator; see, for instance, \cite{GazzolaGrunauSweers}. Since
\[
\Delta^{2}u=F
\qquad \hbox{in } \mathbb R^{4}
\]
in the sense of distributions, and \(F\in L^{s}_{\operatorname{loc}}(\mathbb R^{4})\) for every \(s>1\), the interior \(L^{s}\)-estimates for \(\Delta^{2}\) imply
\[
u\in W^{4,s}_{\operatorname{loc}}(\mathbb R^{4})
\qquad \hbox{for every } 1<s<+\infty .
\]
Taking \(s>4\), the Sobolev embedding theorem gives
\[
u\in C^{3,\gamma}_{\operatorname{loc}}(\mathbb R^{4})
\qquad \hbox{for some } \gamma\in(0,1).
\]
Consequently,
\[
e^{u}\in C^{3,\gamma}_{\operatorname{loc}}(\mathbb R^{4}).
\]

We finally bootstrap. Let \(K\Subset\mathbb R^{4}\), and choose \(\eta\in C_{c}^{\infty}(\mathbb R^{4})\) such that \(\eta\equiv1\) in a neighbourhood of \(K\). We write
\[
V(x)
=
\int_{\mathbb R^{4}}
\frac{\eta(y)e^{u(y)}}{|x-y|^{\alpha}}\,dy
+
\int_{\mathbb R^{4}}
\frac{(1-\eta(y))e^{u(y)}}{|x-y|^{\alpha}}\,dy
=:V_{\operatorname{loc}}(x)+V_{\operatorname{far}}(x).
\]
The far part \(V_{\operatorname{far}}\) is \(C^{\infty}\) in a
neighbourhood of \(K\). Indeed, every derivative of the kernel is smooth on
the separated supports, and its tail is integrable against \(e^{u}\) by
H\"older's inequality and \(e^{u}\in L^{p_{\alpha}}(\mathbb R^{4})\).
The local part \(V_{\operatorname{loc}}\) is the Riesz potential of a
compactly supported \(C^{3,\gamma}\)-function; Lemma
\ref{lem:Riesz-regularity} gives
\[
V_{\operatorname{loc}}\in C^{3,\gamma}_{\operatorname{loc}}(\mathbb R^{4}).
\]
Thus
\[
V\in C^{3,\gamma}_{\operatorname{loc}}(\mathbb R^{4}),
\qquad
F=Ve^{u}\in C^{3,\gamma}_{\operatorname{loc}}(\mathbb R^{4}).
\]
The Schauder estimates for the biharmonic equation yield
\[
u\in C^{7,\gamma}_{\operatorname{loc}}(\mathbb R^{4}).
\]
Repeating the same argument, if \(u\in C^{k,\gamma}_{\operatorname{loc}}(\mathbb R^{4})\), then \(e^{u}\in C^{k,\gamma}_{\operatorname{loc}}(\mathbb R^{4})\), the Riesz potential \(V\in C^{k,\gamma}_{\operatorname{loc}}(\mathbb R^{4})\), and hence
\[
F=Ve^{u}\in C^{k,\gamma}_{\operatorname{loc}}(\mathbb R^{4}).
\]
Applying the Schauder estimates again gives
\[
u\in C^{k+4,\gamma}_{\operatorname{loc}}(\mathbb R^{4}).
\]
By induction,
$
u\in C^{\infty}(\mathbb R^{4}).
$
This completes the proof.
\end{proof}

We next derive the logarithmic representation formula. The following proposition only identifies the biharmonic remainder; no classification of this remainder is claimed without the normality assumption.

\begin{proposition}\label{prop:nonlinear-representation}
Let \(u\) be as in Lemma~\ref{lem:classification-regularity}. Then there exists an entire biharmonic function \(P\) satisfying
\[
\Delta^{2}P=0 \qquad \hbox{in } \mathbb R^{4},
\]
such that
\begin{equation}\label{eq:nonlinear-rep-normalized}
u(x)
=
-\frac{1}{8\pi^{2}}
\int_{\mathbb R^{4}}
\log\frac{|x-y|}{1+|y|}\,F(y)\,dy
+P(x).
\end{equation}
In particular, if the logarithmic moment
\[
\int_{\mathbb R^{4}}\log(1+|y|)F(y)\,dy
\]
is finite, then \eqref{eq:nonlinear-rep-normalized} can be rewritten as
\begin{equation}\label{eq:nonlinear-rep}
u(x)
=
-\frac{1}{8\pi^{2}}
\int_{\mathbb R^{4}}\log|x-y|F(y)\,dy
+\widetilde P(x),
\end{equation}
where \(\widetilde P\) is again an entire biharmonic function. The normality assumption \eqref{eq:normality} is precisely the requirement that this biharmonic remainder be constant.
\end{proposition}

\begin{proof}
By Lemma~\ref{lem:classification-regularity},
\[
F(x)=
\left(
\int_{\mathbb R^{4}}\frac{e^{u(y)}}{|x-y|^{\alpha}}\,dy
\right)e^{u(x)}
\]
belongs to \(L^{1}(\mathbb R^{4})\). Define
\[
v(x):=
-\frac{1}{8\pi^{2}}
\int_{\mathbb R^{4}}
\log\frac{|x-y|}{1+|y|}\,F(y)\,dy .
\]
We first check that \(v\) is well-defined for every fixed \(x\in\mathbb R^{4}\). Let \(R_{x}:=2(1+|x|)\). We split
\[
\mathbb R^{4}
=
B_{1}(x)
\cup
\bigl(B_{R_{x}}(0)\setminus B_{1}(x)\bigr)
\cup
\bigl(\mathbb R^{4}\setminus B_{R_{x}}(0)\bigr).
\]
On \(B_{1}(x)\), the logarithmic singularity is locally integrable. Since \(F\in L^{q}_{\operatorname{loc}}(\mathbb R^{4})\) for some \(q>1\), as follows from the proof of Lemma~\ref{lem:classification-regularity}, H\"older's inequality gives
\[
\int_{B_{1}(x)}|\log |x-y||F(y)\,dy<+\infty.
\]
On \(B_{R_{x}}(0)\setminus B_{1}(x)\), the logarithmic factor is bounded and \(F\in L^{1}_{\operatorname{loc}}(\mathbb R^{4})\). Finally, if \(|y|>R_x\), then
\[
\frac{|y|}{2}\leq |x-y|\leq \frac32 |y|,
\]
and therefore
\[
\left|\log\frac{|x-y|}{1+|y|}\right|\leq C.
\]
Since \(F\in L^{1}(\mathbb R^{4})\), the integral defining \(v\) is finite.

We next prove that
\[
\Delta^{2}v=F
\qquad \hbox{in }\mathcal D'(\mathbb R^{4}).
\]
Let \(\varphi\in C_{c}^{\infty}(\mathbb R^{4})\). Since
\[
\Delta^{2}\left(-\frac{1}{8\pi^{2}}\log|x|\right)=\delta_{0}
\qquad \hbox{in }\mathcal D'(\mathbb R^{4})
\]
and
\[
\int_{\mathbb R^{4}}\Delta^{2}\varphi(x)\,dx=0,
\]
Fubini's theorem gives
\[
\begin{aligned}
\int_{\mathbb R^{4}}v(x)\Delta^{2}\varphi(x)\,dx
&=
-\frac{1}{8\pi^{2}}
\int_{\mathbb R^{4}}F(y)
\int_{\mathbb R^{4}}
\log\frac{|x-y|}{1+|y|}
\Delta^{2}\varphi(x)\,dxdy  \\
&=
-\frac{1}{8\pi^{2}}
\int_{\mathbb R^{4}}F(y)
\int_{\mathbb R^{4}}\log|x-y|\Delta^{2}\varphi(x)\,dxdy  \\
&=
\int_{\mathbb R^{4}}F(y)\varphi(y)\,dy .
\end{aligned}
\]
Thus \(\Delta^{2}v=F\). Since \(u\) is a distributional solution of \(\Delta^{2}u=F\), the difference
\[
P:=u-v
\]
satisfies \(\Delta^{2}P=0\) in \(\mathbb R^{4}\). Weyl's lemma for the biharmonic operator yields \(P\in C^{\infty}(\mathbb R^{4})\).

If \(\int_{\mathbb R^{4}}\log(1+|y|)F(y)\,dy<+\infty\), then
\[
\begin{aligned}
u(x)
&=
-\frac{1}{8\pi^{2}}
\int_{\mathbb R^{4}}\log|x-y|F(y)\,dy
+
\frac{1}{8\pi^{2}}
\int_{\mathbb R^{4}}\log(1+|y|)F(y)\,dy
+P(x) \\
&=
-\frac{1}{8\pi^{2}}
\int_{\mathbb R^{4}}\log|x-y|F(y)\,dy
+\widetilde P(x),
\end{aligned}
\]
where
\[
\widetilde P(x):=
P(x)+
\frac{1}{8\pi^{2}}
\int_{\mathbb R^{4}}\log(1+|y|)F(y)\,dy.
\]
Then \(\Delta^{2}\widetilde P=0\). Finally, the normality condition \eqref{eq:normality} says exactly that in this unnormalized representation the entire biharmonic remainder is a constant. This completes the proof.
\end{proof}

\begin{remark}
For fourth-order and higher-order Liouville-type equations, finite volume or finite mass alone may allow a nonconstant polynomial remainder in the logarithmic representation. In the present paper we do not use, and do not need, a general theorem asserting that the finite-mass condition by itself forces the remainder \(P\) in Proposition~\ref{prop:nonlinear-representation} to be a polynomial of degree at most two. The classification theorem is deliberately stated for normal finite-mass solutions, for which the biharmonic remainder in \eqref{eq:nonlinear-rep} is assumed to be constant. This is the precise point at which the normality assumption enters the proof.
\end{remark}

For the remainder of this section, unless otherwise stated, \(u\) denotes a normal finite-mass solution. Put
\[
M:=\int_{\mathbb R^{4}}F(x)\,dx,
\qquad
\gamma:=\frac{M}{8\pi^{2}}.
\]
The next lemma supplies the precise far-field information used in the Pohozaev identity. Its proof keeps the moving near-field contribution separate until a uniform local estimate has been obtained.

\begin{lemma}\label{lem:normal-asymptotics}
Let \(u\) be a normal finite-mass solution of \eqref{eq:main}. Then
\begin{equation}\label{eq:strict-mass-lower}
\frac{p_{\alpha}M}{8\pi^{2}}>4,
\qquad p_{\alpha}:=\frac{8}{8-\alpha}.
\end{equation}
There exists \(\varepsilon_{1}>0\) such that
\begin{equation}\label{eq:F-tail-decay}
F(x)\leq C(1+|x|)^{-4-\varepsilon_{1}}
\qquad\hbox{for }|x|\geq1.
\end{equation}
Moreover, as \(|x|\to+\infty\),
\begin{equation}\label{eq:u-asymptotic}
u(x)=-\frac{M}{8\pi^{2}}\log|x|+O(1),
\end{equation}
\begin{equation}\label{eq:du-asymptotic}
\partial_{r}u(x)=-\frac{M}{8\pi^{2}}\frac1{|x|}+o(|x|^{-1}),
\end{equation}
\begin{equation}\label{eq:drr-u-asymptotic}
\partial_{rr}u(x)=\frac{M}{8\pi^{2}}\frac1{|x|^{2}}+o(|x|^{-2}),
\end{equation}
\begin{equation*}
\Delta u(x)=-\frac{M}{4\pi^{2}}\frac1{|x|^{2}}+o(|x|^{-2}),
\end{equation*}
and
\begin{equation}\label{eq:dr-lap-u-asymptotic}
\partial_{r}\Delta u(x)=\frac{M}{2\pi^{2}}\frac1{|x|^{3}}+o(|x|^{-3}).
\end{equation}
All the remainder estimates in \eqref{eq:du-asymptotic}--\eqref{eq:dr-lap-u-asymptotic} are uniform with respect to the angular variable \(x/|x|\).
\end{lemma}

\begin{proof}
By normality,
\begin{equation}\label{eq:normal-representation-asymptotic}
u(x)=-\frac{1}{8\pi^{2}}
\int_{\mathbb R^{4}}\log|x-y|F(y)\,dy+C.
\end{equation}
We first record the logarithmic moment. Fix \(x_{0}\in\mathbb R^{4}\). On a neighbourhood of \(x_{0}\), the negative part of \(\log|x_{0}-y|\) is integrable against \(F\), because \(F\in L^{q}_{\operatorname{loc}}\) for every finite \(q>1\). For \(|y|\) large,
\(\log|x_{0}-y|\geq\log(2+|y|)-C\). Since the integral in \eqref{eq:normal-representation-asymptotic} is finite, it follows that
\begin{equation}\label{eq:logarithmic-moment}
\int_{\mathbb R^{4}}\log(2+|y|)F(y)\,dy<+\infty.
\end{equation}
The elementary bound
\[
|x-y|\leq(1+|x|)(1+|y|)
\]
and \eqref{eq:logarithmic-moment} give
\begin{equation}\label{eq:lower-u-preliminary}
u(x)\geq-\gamma\log(1+|x|)-C.
\end{equation}
If \(p_{\alpha}\gamma\leq4\), then \eqref{eq:lower-u-preliminary} would imply
\(e^{p_{\alpha}u(x)}\geq C(1+|x|)^{-p_{\alpha}\gamma}\), which is not integrable in \(\mathbb R^{4}\). This contradicts \eqref{eq:finite-mass} and proves \eqref{eq:strict-mass-lower}.

We next obtain an almost sharp upper bound without assuming uniform local control of \(F\) at infinity. Fix
\begin{equation}\label{eq:epsilon-choice-asymptotic}
0<\varepsilon<\gamma-\frac{8-\alpha}{2}.
\end{equation}
Choose \(s>\max\{1,4/(4-\alpha)\}\). Since \(F\in L^{1}(\mathbb R^{4})\), one can choose \(R>1\) such that, with
\[
m_{R}:=\int_{\mathbb R^{4}\setminus B_{R}}F(y)\,dy,
\qquad
\gamma_{R}:=\frac{1}{8\pi^{2}}\int_{B_{R}}F(y)\,dy,
\]
we have
\begin{equation}\label{eq:tail-small-choice}
\gamma_{R}>\gamma-\varepsilon,
\qquad
\frac{s m_{R}}{8\pi^{2}}<4.
\end{equation}
Let \(|x|>2R+8\). Define
\[
v_{x}(z):=-\frac{1}{8\pi^{2}}
\int_{B_{4}(x)}\log|z-y|F(y)\,dy,
\qquad
h_{x}:=u-v_{x}.
\]
Then \(\Delta^{2}h_{x}=0\) in \(B_{4}(x)\). For \(z\in B_{2}(x)\), the set \(B_{R}\) is disjoint from \(B_{4}(x)\), and
\[
\log|z-y|=\log|x|+O_{R}(1)
\qquad\hbox{uniformly for }y\in B_{R}.
\]
If \(y\notin B_{4}(x)\), then \(|z-y|\geq2\). The contribution of
\(\mathbb R^{4}\setminus(B_{R}\cup B_{4}(x))\) to \(h_x(z)\) is
nonpositive. The fixed additive constant in \eqref{eq:normal-representation-asymptotic}
and the bounded error on \(B_R\) can be absorbed into \(C_R\). Hence
\begin{equation}\label{eq:hx-upper}
h_{x}(z)\leq-\gamma_{R}\log|x|+C_{R}
\qquad\hbox{for }z\in B_{2}(x).
\end{equation}

Put \(m_{x}:=\int_{B_{4}(x)}F\leq m_{R}\). If \(m_x>0\), let
\(d\nu_x=F(y)m_x^{-1}dy\) on \(B_4(x)\). For \(z\in B_2(x)\), Jensen's inequality yields
\[
\begin{aligned}
e^{s v_x(z)}
&\leq
\exp\left(\frac{s m_x}{8\pi^{2}}
\int_{B_4(x)}\log\frac{6}{|z-y|}\,d\nu_x(y)\right)\leq
\int_{B_4(x)}
\left(\frac{6}{|z-y|}\right)^{s m_x/(8\pi^{2})}
\,d\nu_x(y).
\end{aligned}
\]
The same conclusion is immediate if \(m_x=0\). Integrating in \(z\), using \eqref{eq:tail-small-choice}, and then using \eqref{eq:hx-upper}, we obtain
\begin{equation}\label{eq:uniform-moving-exponential}
\int_{B_{2}(x)}e^{s u(z)}\,dz
\leq C_{R}|x|^{-s\gamma_{R}}.
\end{equation}

For \(z\in B_{1}(x)\), split the Riesz potential as
\[
V(z)=\int_{B_{2}(x)}\frac{e^{u(y)}}{|z-y|^{\alpha}}\,dy
+\int_{\mathbb R^{4}\setminus B_{2}(x)}
\frac{e^{u(y)}}{|z-y|^{\alpha}}\,dy.
\]
The first term is bounded by \(C_R|x|^{-\gamma_R}\) by
\eqref{eq:uniform-moving-exponential} and the choice of \(s\). For the second term, let \(p_{\alpha}'=8/\alpha\). Since \(|z-y|\geq1\) there and \(\alpha p_{\alpha}'=8>4\), H\"older's inequality and \eqref{eq:finite-mass} give a bound independent of \(x\). Thus
\begin{equation}\label{eq:V-moving-bound}
\|V\|_{L^{\infty}(B_{1}(x))}\leq C_{R}.
\end{equation}
Combining \eqref{eq:uniform-moving-exponential} and \eqref{eq:V-moving-bound},
\begin{equation*}
\|F\|_{L^{s}(B_{1}(x))}\leq C_{R}|x|^{-\gamma_{R}}.
\end{equation*}
Moreover, \(v_x\geq-m_R\log 6/(8\pi^{2})\), while the preceding Jensen estimate controls its positive part. Hence \(\|v_x\|_{L^{s}(B_1(x))}\leq C_R\). The interior \(W^{4,s}\)-estimate for
\(\Delta^{2}v_x=F\) in \(B_1(x)\) now gives
\[
\|v_x\|_{L^{\infty}(B_{1/2}(x))}
\leq C\left(\|v_x\|_{L^{s}(B_1(x))}
+\|F\|_{L^{s}(B_1(x))}\right)\leq C_R.
\]
Together with \eqref{eq:hx-upper}, this proves
\begin{equation*}
u(x)\leq-(\gamma-\varepsilon)\log|x|+C_{\varepsilon}
\qquad\hbox{for }|x|\geq2.
\end{equation*}

Choose \(\varepsilon\) in \eqref{eq:epsilon-choice-asymptotic} and set
\(\theta:=\gamma-\varepsilon>(8-\alpha)/2\). Then
\[
e^{u(x)}\leq C(1+|x|^{2})^{-\theta/2}.
\]
Since \(\theta+\alpha>4\), Lemma~\ref{lem:Riesz-estimate} applies. If \(\theta<4\), it gives
\[
F(x)\leq C(1+|x|^{2})^{(4-\alpha-2\theta)/2}.
\]
If \(\theta=4\), the same estimate holds with an additional logarithm; if \(\theta>4\), it gives
\(F(x)\leq C(1+|x|^{2})^{-(\alpha+\theta)/2}\). In the first case the corresponding power of \(|x|\) is
\(4-\alpha-2\theta<-4\). In the borderline case the logarithm is
absorbed by a smaller negative power. In the last case the power is
\(-\alpha-\theta<-4\). Thus, after decreasing a positive exponent if
necessary, \eqref{eq:F-tail-decay} follows.

We can now return to the logarithmic potential without a moving singularity problem. Let \(R=|x|\) and \(e=x/R\). Split
\[
\int_{\mathbb R^{4}}\log\frac{|x-y|}{R}F(y)\,dy
=I_{1}(x)+I_{2}(x),
\]
where \(I_1\) is over \(|y|\leq R/2\). On this region
\(|\log|e-y/R||\leq C\), so \(|I_1(x)|\leq CM\). By \eqref{eq:F-tail-decay} and the change of variables \(y=Rz\),
\[
|I_2(x)|
\leq C R^{-\varepsilon_{1}}
\int_{|z|>1/2}|\log|e-z||\,|z|^{-4-\varepsilon_{1}}\,dz
\leq C R^{-\varepsilon_{1}}.
\]
The last integral is finite uniformly for \(e\in\mathbb S^{3}\). Formula
\eqref{eq:normal-representation-asymptotic} therefore proves
\eqref{eq:u-asymptotic}.

Finally, \eqref{eq:F-tail-decay} and local smoothness imply that the first three derivatives of the logarithmic potential may be taken under the integral sign. For \(m=1,2,3\),
\begin{equation}\label{eq:scaled-tail-kernels}
R^{m}\int_{|y|>R/2}|x-y|^{-m}F(y)\,dy
\leq C R^{-\varepsilon_{1}}
\int_{|z|>1/2}|e-z|^{-m}|z|^{-4-\varepsilon_{1}}\,dz
=o(1).
\end{equation}
On \(|y|\leq R/2\), the relevant scaled kernels are uniformly bounded because
\(|e-y/R|\geq1/2\). For each fixed \(A>0\), their limits on \(|y|\leq A\)
are uniform in \(e=x/R\in\mathbb S^{3}\):
\[
R\frac{(x-y)\cdot x}{|x-y|^{2}|x|}\longrightarrow1,
\]
\[
R^{2}\partial_{rr}\log|x-y|
=R^{2}\left(
\frac{1}{|x-y|^{2}}
-2\frac{((x-y)\cdot e)^{2}}{|x-y|^{4}}
\right)\longrightarrow-1,
\]
\[
R^{2}\Delta\log|x-y|\longrightarrow2,
\qquad
R^{3}\partial_{r}\Delta\log|x-y|
=-4R^{3}\frac{(x-y)\cdot x}{|x|\,|x-y|^{4}}
\longrightarrow-4.
\]
On the remaining region \(A<|y|\leq R/2\), the uniform boundedness of these
kernels gives an error bounded by
\[
C\int_{|y|>A}F(y)\,dy.
\]
We first let \(R\to+\infty\) with \(A\) fixed and then let \(A\to+\infty\).
This proves uniform convergence of all interior contributions. The tail
estimate \eqref{eq:scaled-tail-kernels} applies to the complementary region.
Combining these facts with \eqref{eq:normal-representation-asymptotic} yields
\eqref{eq:du-asymptotic}--\eqref{eq:dr-lap-u-asymptotic}, including
\eqref{eq:drr-u-asymptotic}, with uniform remainders in the angular variable.
\end{proof}

The preceding asymptotics allow us to compute the conformal mass.

\begin{lemma}\label{lem:pohozaev}
Let \(u\) be a smooth normal finite-mass solution of \eqref{eq:main}. Then
\begin{equation}\label{eq:pohozaev}
\int_{\mathbb R^{4}}\int_{\mathbb R^{4}}
\frac{e^{u(x)}e^{u(y)}}{|x-y|^{\alpha}}\,dxdy
=
8\pi^{2}(8-\alpha).
\end{equation}
\end{lemma}

\begin{proof}
For simplicity write
\[
F(x)=
\left(\int_{\mathbb R^{4}}\frac{e^{u(y)}}{|x-y|^{\alpha}}\,dy\right)e^{u(x)},
\qquad
M:=\int_{\mathbb R^{4}}F(x)\,dx.
\]
Since \(u\) is normal, we have
\[
u(x)=
-\frac{1}{8\pi^{2}}
\int_{\mathbb R^{4}}\log|x-y|F(y)\,dy+C.
\]
Lemma~\ref{lem:normal-asymptotics} gives the decay \eqref{eq:F-tail-decay} and the following expansions as \(|x|\to+\infty\):
\[
u(x)=-\frac{M}{8\pi^{2}}\log|x|+O(1),
\]
\[
\partial_{r}u(x)=-\frac{M}{8\pi^{2}}\frac{1}{|x|}+o(|x|^{-1}),
\qquad
\partial_{rr}u(x)=\frac{M}{8\pi^{2}}\frac{1}{|x|^{2}}+o(|x|^{-2}),
\]
\[
\Delta u(x)=-\frac{M}{4\pi^{2}}\frac{1}{|x|^{2}}+o(|x|^{-2}),
\]
and
\[
\partial_{r}\Delta u(x)=
\frac{M}{2\pi^{2}}\frac{1}{|x|^{3}}+o(|x|^{-3}).
\]
Here and below \(\partial_{r}\) denotes the radial derivative on \(\partial B_{R}\).

Let \(B_{R}=B_{R}(0)\). Multiplying \eqref{eq:main} by \(x\cdot\nabla u\) and integrating over \(B_{R}\), we obtain
\begin{equation}\label{eq:pohozaev-start}
\int_{B_{R}}\Delta^{2}u(x)(x\cdot\nabla u(x))\,dx
=
\int_{B_{R}}F(x)(x\cdot\nabla u(x))\,dx.
\end{equation}
We first compute the limit of the left-hand side. Put \(w=\Delta u\). By two integrations by parts,
\[
\begin{aligned}
\int_{B_{R}}\Delta w(x)(x\cdot\nabla u(x))\,dx
&=
\int_{\partial B_{R}}\partial_{\nu}w(x)(x\cdot\nabla u(x))\,dS  \\
&\quad
-
\int_{\partial B_{R}}w(x)\partial_{\nu}(x\cdot\nabla u(x))\,dS
+
\int_{B_{R}}w(x)\Delta(x\cdot\nabla u(x))\,dx.
\end{aligned}
\]
Since
\[
\Delta(x\cdot\nabla u)=x\cdot\nabla\Delta u+2\Delta u=x\cdot\nabla w+2w,
\]
and the dimension is four, we have
\[
\int_{B_{R}}w(x)x\cdot\nabla w(x)\,dx
=
\frac{R}{2}\int_{\partial B_{R}}w^{2}\,dS
-2\int_{B_{R}}w^{2}\,dx.
\]
Thus the interior terms cancel and
\begin{equation}\label{eq:boundary-pohozaev}
\begin{aligned}
\int_{B_{R}}\Delta^{2}u(x)(x\cdot\nabla u(x))\,dx
={}&
\int_{\partial B_{R}}
\left[
\partial_{\nu}\Delta u(x)(x\cdot\nabla u(x))
-
\Delta u(x)\partial_{\nu}(x\cdot\nabla u(x))
\right]dS\\
&+
\frac{R}{2}\int_{\partial B_{R}}(\Delta u(x))^{2}\,dS.
\end{aligned}
\end{equation}
Using the asymptotic estimates above, on \(\partial B_{R}\) we have
\[
x\cdot\nabla u=-\frac{M}{8\pi^{2}}+o(1),
\qquad
\partial_{\nu}(x\cdot\nabla u)
=\partial_{r}(r\partial_{r}u)
=\partial_{r}u+R\partial_{rr}u=o(R^{-1}),
\]
\[
\Delta u=-\frac{M}{4\pi^{2}}R^{-2}+o(R^{-2}),
\qquad
\partial_{\nu}\Delta u=\frac{M}{2\pi^{2}}R^{-3}+o(R^{-3}).
\]
Since \(|\partial B_{R}|=2\pi^{2}R^{3}\), passing to the limit in \eqref{eq:boundary-pohozaev} gives
\begin{equation}\label{eq:left-pohozaev-limit}
\lim_{R\to+\infty}
\int_{B_{R}}\Delta^{2}u(x)(x\cdot\nabla u(x))\,dx
=
-\frac{M^{2}}{16\pi^{2}}.
\end{equation}

We now compute the right-hand side of \eqref{eq:pohozaev-start}. Define
\[
I(u):=
\int_{\mathbb R^{4}}\int_{\mathbb R^{4}}
\frac{e^{u(x)}e^{u(y)}}{|x-y|^{\alpha}}\,dxdy.
\]
Then \(M=I(u)\). For
\[
u_{\lambda}(x)=u(\lambda x)+\frac{8-\alpha}{2}\log\lambda,
\]
the invariance of the double integral gives \(I(u_{\lambda})=I(u)\). The
asymptotic estimate \eqref{eq:du-asymptotic}, together with local smoothness,
shows that
\[
\sup_{z\in\mathbb R^{4}}|z\cdot\nabla u(z)|<+\infty.
\]
Consequently, for \(\lambda\) in a compact neighbourhood of one, the absolute
value of the derivative of the double-integral integrand is integrable. Indeed,
after the change of variables \(z=\lambda x\), \(w=\lambda y\), its integral is
bounded by a constant multiple of
\[
\int_{\mathbb R^{4}}\int_{\mathbb R^{4}}
\frac{e^{u(z)}e^{u(w)}}{|z-w|^{\alpha}}
\bigl(1+|z\cdot\nabla u(z)|+|w\cdot\nabla u(w)|\bigr)\,dzdw<+\infty.
\]
Differentiation under the integral sign at \(\lambda=1\) is therefore
legitimate and gives
\[
\begin{aligned}
0
={}&\left.\frac{d}{d\lambda}I(u_{\lambda})\right|_{\lambda=1}\\
={}&
\int_{\mathbb R^{4}}\int_{\mathbb R^{4}}
\frac{e^{u(x)+u(y)}}{|x-y|^{\alpha}}
\left[x\cdot\nabla u(x)+y\cdot\nabla u(y)+(8-\alpha)\right]\,dxdy\\
={}&
2\int_{\mathbb R^{4}}F(x)(x\cdot\nabla u(x))\,dx
+(8-\alpha)I(u).
\end{aligned}
\]
Therefore
\begin{equation}\label{eq:right-pohozaev-limit}
\int_{\mathbb R^{4}}F(x)(x\cdot\nabla u(x))\,dx
=
-\frac{8-\alpha}{2}M.
\end{equation}
By Lemma~\ref{lem:normal-asymptotics}, \(x\cdot\nabla u(x)=O(1)\) as \(|x|\to+\infty\). Together with \(F\in L^{1}(\mathbb R^{4})\), and in fact the stronger decay \eqref{eq:F-tail-decay}, this justifies passing to the limit by dominated convergence:
\[
\int_{B_{R}}F(x)(x\cdot\nabla u(x))\,dx
\to
\int_{\mathbb R^{4}}F(x)(x\cdot\nabla u(x))\,dx.
\]
Combining \eqref{eq:left-pohozaev-limit} and \eqref{eq:right-pohozaev-limit}, we obtain
\[
-\frac{M^{2}}{16\pi^{2}}
=
-\frac{8-\alpha}{2}M.
\]
Since \(M>0\), we conclude
\[
M=8\pi^{2}(8-\alpha).
\]
Since \(M=I(u)\), this proves \eqref{eq:pohozaev}.
\end{proof}

We next verify the growth and integrability assumptions required by the published higher-order Choquard classification theorem.

\begin{proposition}\label{prop:reduction-HN}
Let \(u\) be a normal finite-mass solution of \eqref{eq:main}, and put
\[
b_{\alpha}:=\frac{8-\alpha}{2},
\qquad
v(x):=\frac{1}{b_{\alpha}}
\left(u(x)-\frac12\log b_{\alpha}\right).
\]
Then \(v\) is a classical solution of
\begin{equation}\label{eq:HN-normalized}
\Delta^{2}v=
\left(\int_{\mathbb R^{4}}
\frac{e^{b_{\alpha}v(y)}}{|x-y|^{\alpha}}\,dy\right)
e^{b_{\alpha}v(x)}
\qquad\hbox{in }\mathbb R^{4},
\end{equation}
and it satisfies
\begin{equation*}
v(x)=o(|x|^{2}),
\qquad
\int_{\mathbb R^{4}}e^{b_{\alpha}v(x)}\,dx<+\infty,
\end{equation*}
together with
\begin{equation*}
\int_{\mathbb R^{4}}\int_{\mathbb R^{4}}
\frac{e^{b_{\alpha}v(x)}e^{b_{\alpha}v(y)}}{|x-y|^{\alpha}}\,dxdy
<+\infty.
\end{equation*}
Consequently, Niu's published classification theorem \cite{NiuClassification} applies.
\end{proposition}

\begin{proof}
The regularity follows from Lemma~\ref{lem:classification-regularity}. Since
\[
e^{u}=b_{\alpha}^{1/2}e^{b_{\alpha}v},
\qquad
\Delta^{2}u=b_{\alpha}\Delta^{2}v,
\]
substitution into \eqref{eq:main} gives \eqref{eq:HN-normalized}.

By Lemma~\ref{lem:pohozaev},
\[
M:=\int_{\mathbb R^{4}}F(x)\,dx
=8\pi^{2}(8-\alpha).
\]
Lemma~\ref{lem:normal-asymptotics} therefore yields
\begin{equation*}
u(x)=-(8-\alpha)\log|x|+O(1)
\qquad\hbox{as }|x|\to+\infty.
\end{equation*}
Since \(8-\alpha>4\), it follows that \(e^{u}\in L^{1}(\mathbb R^{4})\), and hence
\[
\int_{\mathbb R^{4}}e^{b_{\alpha}v}\,dx
=b_{\alpha}^{-1/2}\int_{\mathbb R^{4}}e^{u}\,dx<+\infty.
\]
The same asymptotic formula gives \(v(x)=O(\log|x|)=o(|x|^{2})\). Finally,
\[
\int_{\mathbb R^{4}}\int_{\mathbb R^{4}}
\frac{e^{b_{\alpha}v(x)}e^{b_{\alpha}v(y)}}{|x-y|^{\alpha}}\,dxdy
=b_{\alpha}^{-1}M<+\infty.
\]
Thus all assumptions of the cited classification theorem are verified.
\end{proof}

\begin{proof}[Proof of Theorem~\ref{thm:classification}]
By Proposition~\ref{prop:reduction-HN}, Niu's classification theorem \cite{NiuClassification} implies that there exist \(\mu>0\), \(\zeta\in\mathbb R^{4}\), and \(c>0\) such that
\[
u(x)=\frac{8-\alpha}{2}\log
\left(\frac{c\mu}{1+\mu^{2}|x-\zeta|^{2}}\right).
\]
Substitution into \eqref{eq:main}, together with Lemma~\ref{lem:convolution-identity}, gives
\[
c^{8-\alpha}
=\frac{12(8-\alpha)(6-\alpha)(4-\alpha)}{\pi^{2}},
\]
so \(c=C_{\alpha}\). This proves the asserted form of the solution.

For the conformal mass, the change of variables \(z=\mu(x-\zeta)\) gives
\[
\int_{\mathbb R^{4}}
e^{\frac{8}{8-\alpha}U_{\mu,\zeta}(x)}\,dx
=C_{\alpha}^{4}
\int_{\mathbb R^{4}}\frac{dz}{(1+|z|^{2})^{4}}
=\frac{\pi^{2}}{6}C_{\alpha}^{4}.
\]
The double integral is the exact mass obtained in Lemma~\ref{lem:pohozaev}, namely \(8\pi^{2}(8-\alpha)\). This proves \eqref{eq:mass-identity} and \eqref{eq:double-identity}.
\end{proof}

\begin{remark}
The use of \cite{HuangNiu,NiuClassification} is deliberate. It separates the already established general moving-spheres classification from the new spectral content of this paper. What is proved here on the nonlinear side is the implication from the single conformal mass condition and normal distributional representation to the classical growth and integrability assumptions required by that theorem.
\end{remark}

\section{Conformal realization of the linearized operator}

Define the spherical Riesz operator
\begin{equation*}
\mathcal R_{\alpha}\Phi(\xi)
:=\int_{\mathbb S^{4}}
\frac{\Phi(\eta)}{|\xi-\eta|^{\alpha}}\,d\eta.
\end{equation*}

\begin{proposition}\label{prop:Riesz-operator}
The operator \(\mathcal R_{\alpha}\) is positive, compact and self-adjoint on \(L^{2}(\mathbb S^{4})\). Its restriction to \(\mathcal H_{k}\) is multiplication by \(\mu_{k}(\alpha)\).
\end{proposition}

\begin{proof}
The kernel is symmetric. Since
\[
\sup_{\xi\in\mathbb S^{4}}
\int_{\mathbb S^{4}}|\xi-\eta|^{-\alpha}\,d\eta
=\mu_{0}(\alpha)<+\infty,
\]
Schur's test gives boundedness on \(L^{2}\), and symmetry gives self-adjointness. Lemma~\ref{lem:Funk-Hecke} diagonalizes the operator on the orthogonal decomposition into spherical harmonics. Since \(\mu_k(\alpha)>0\) for every \(k\), the operator is positive. Formula \eqref{eq:mu-k} and the gamma quotient asymptotic imply \(\mu_{k}(\alpha)=O(k^{\alpha-4})\to0\). Hence the finite-rank truncations on \(\bigoplus_{k=0}^{K}\mathcal H_{k}\) converge in operator norm to \(\mathcal R_{\alpha}\), proving compactness.
\end{proof}

Set
\begin{equation}\label{eq:kappa-alpha}
\kappa_{\alpha}:=C_{\alpha}^{8-\alpha}2^{-(8-\alpha)}.
\end{equation}
A direct use of \eqref{eq:C-alpha}, \eqref{eq:mu0-mu1} and \eqref{eq:sphere-volume} gives
\begin{equation}\label{eq:kappa-identities}
\kappa_{\alpha}\mu_{0}(\alpha)=3(8-\alpha),
\qquad
\frac{4\pi^{2}(8-\alpha)}{|\mathbb S^{4}|\mu_{0}(\alpha)}
=\frac{\kappa_{\alpha}}{2}.
\end{equation}

The Euclidean and spherical linearized operators are related by the following identity.

\begin{proposition}\label{prop:conformal-linearized}
Let \(\Phi\) be smooth on \(\mathbb S^{4}\), and put \(\phi=S^{*}\Phi\). Then the Euclidean linearized operator at \(U\) is transformed into \(\mathcal L_{\alpha}\):
\begin{equation}\label{eq:conformal-linearized-identity}
\begin{aligned}
&\Delta^{2}\phi(x)
-
\left(\int_{\mathbb R^{4}}
\frac{e^{U(y)}\phi(y)}{|x-y|^{\alpha}}\,dy\right)e^{U(x)}\\
&\quad-
\left(\int_{\mathbb R^{4}}
\frac{e^{U(y)}}{|x-y|^{\alpha}}\,dy\right)e^{U(x)}\phi(x)
=
\omega(x)^{4}S^{*}(\mathcal L_{\alpha}\Phi)(x).
\end{aligned}
\end{equation}
Moreover, if
\begin{equation*}
\begin{aligned}
\mathcal Q_{U}(\phi)
:={}&\int_{\mathbb R^{4}}|\Delta\phi|^{2}\,dx
-
\int_{\mathbb R^{4}}\int_{\mathbb R^{4}}
\frac{e^{U(x)}\phi(x)e^{U(y)}\phi(y)}{|x-y|^{\alpha}}\,dxdy\\
&-
\int_{\mathbb R^{4}}
\left(\int_{\mathbb R^{4}}
\frac{e^{U(y)}}{|x-y|^{\alpha}}\,dy\right)
e^{U(x)}\phi(x)^{2}\,dx,
\end{aligned}
\end{equation*}
then
\begin{equation}\label{eq:quadratic-form-conformal}
\mathcal Q_{U}(S^{*}\Phi)
=\langle\mathcal L_{\alpha}\Phi,\Phi\rangle_{L^{2}(\mathbb S^{4})}.
\end{equation}
\end{proposition}

\begin{proof}
From \eqref{eq:normalized-bubble} and \eqref{eq:omega},
\begin{equation}\label{eq:eU-omega}
e^{U(x)}
=C_{\alpha}^{\frac{8-\alpha}{2}}2^{-\frac{8-\alpha}{2}}
\omega(x)^{\frac{8-\alpha}{2}}.
\end{equation}
Let \(\xi=Sx\), \(\eta=Sy\). Using \eqref{eq:stereo-identities},
\[
\begin{aligned}
\int_{\mathbb R^{4}}
\frac{e^{U(y)}\phi(y)}{|x-y|^{\alpha}}\,dy
&=C_{\alpha}^{\frac{8-\alpha}{2}}2^{-\frac{8-\alpha}{2}}
\omega(x)^{\alpha/2}
\int_{\mathbb S^{4}}
\frac{\Phi(\eta)}{|\xi-\eta|^{\alpha}}\,d\eta=C_{\alpha}^{\frac{8-\alpha}{2}}2^{-\frac{8-\alpha}{2}}
\omega(x)^{\alpha/2}
S^{*}(\mathcal R_{\alpha}\Phi)(x).
\end{aligned}
\]
The same formula with \(\Phi\equiv1\) gives
\[
\int_{\mathbb R^{4}}
\frac{e^{U(y)}}{|x-y|^{\alpha}}\,dy
=C_{\alpha}^{\frac{8-\alpha}{2}}2^{-\frac{8-\alpha}{2}}
\mu_{0}(\alpha)\omega(x)^{\alpha/2}.
\]
Multiplication by \eqref{eq:eU-omega} produces the common factor \(\kappa_{\alpha}\omega^{4}\). Combining this with the Paneitz covariance \eqref{eq:Paneitz-covariance} proves \eqref{eq:conformal-linearized-identity}.

The conformal energy identity
\[
\int_{\mathbb R^{4}}|\Delta S^{*}\Phi|^{2}\,dx
=\int_{\mathbb S^{4}}\Phi P_{4}^{\mathbb S^{4}}\Phi\,d\xi
\]
follows from \eqref{eq:Paneitz-covariance}. A second use of \eqref{eq:stereo-identities} gives
\[
\int_{\mathbb R^{4}}\int_{\mathbb R^{4}}
\frac{e^{U(x)}\phi(x)e^{U(y)}\phi(y)}{|x-y|^{\alpha}}\,dxdy
=\kappa_{\alpha}\langle\mathcal R_{\alpha}\Phi,\Phi\rangle,
\]
and
\[
\int_{\mathbb R^{4}}
\left(\int_{\mathbb R^{4}}
\frac{e^{U(y)}}{|x-y|^{\alpha}}\,dy\right)
e^{U(x)}\phi(x)^{2}\,dx
=\kappa_{\alpha}\mu_{0}(\alpha)\|\Phi\|_{2}^{2}.
\]
This proves \eqref{eq:quadratic-form-conformal}.
\end{proof}

We now remove the possible defect at the stereographic pole. The first step records the weak form of the conformal transfer. This avoids applying the pointwise identity of Proposition~\ref{prop:conformal-linearized} directly to a merely weighted \(L^{2}\)-solution.

\begin{lemma}\label{lem:weak-conformal-transfer}
Let \(\phi\in L^{2}_{w}(\mathbb R^{4})\) be a distributional solution of
\eqref{eq:linearized}, and put \(\Phi=S_{*}\phi\). Then
\(\Phi\in L^{2}(\mathbb S^{4})\) and
\begin{equation}\label{eq:sphere-away-pole}
\mathcal L_{\alpha}\Phi=0
\qquad\hbox{in }\mathcal D'(\mathbb S^{4}\setminus\{p\}).
\end{equation}
\end{lemma}

\begin{proof}
Identity \eqref{eq:L2-isometry} gives \(\Phi\in L^{2}(\mathbb S^{4})\). Let
\(\chi\in C_{c}^{\infty}(\mathbb S^{4}\setminus\{p\})\). Then
\(S^{*}\chi\in C_{c}^{\infty}(\mathbb R^{4})\). Denote the Euclidean
linearized distribution by
\[
\mathscr L_{U}\phi
:=
\Delta^{2}\phi
-e^{U}\int_{\mathbb R^{4}}
\frac{e^{U(y)}\phi(y)}{|\,\cdot-y\,|^{\alpha}}\,dy
-e^{U}\phi\int_{\mathbb R^{4}}
\frac{e^{U(y)}}{|\,\cdot-y\,|^{\alpha}}\,dy.
\]
We claim that
\[
\langle\mathcal L_{\alpha}\Phi,\chi\rangle
=
\langle\mathscr L_{U}\phi,S^{*}\chi\rangle.
\]
Indeed, the local term satisfies
\[
\int_{\mathbb R^{4}}\phi\,\Delta^{2}(S^{*}\chi)\,dx
=
\int_{\mathbb R^{4}}\phi\,\omega^{4}
S^{*}\bigl(P_{4}^{\mathbb S^{4}}\chi\bigr)\,dx
=
\int_{\mathbb S^{4}}\Phi P_{4}^{\mathbb S^{4}}\chi\,d\xi.
\]
For the first nonlocal term, the change of variables used in
Proposition~\ref{prop:conformal-linearized}, followed by the self-adjointness
of \(\mathcal R_{\alpha}\), gives
\[
\begin{aligned}
&\int_{\mathbb R^{4}}e^{U(x)}S^{*}\chi(x)
\left(\int_{\mathbb R^{4}}
\frac{e^{U(y)}\phi(y)}{|x-y|^{\alpha}}\,dy\right)dx
=\kappa_{\alpha}
\int_{\mathbb S^{4}}\chi\,\mathcal R_{\alpha}\Phi\,d\xi
=\kappa_{\alpha}
\int_{\mathbb S^{4}}\Phi\,\mathcal R_{\alpha}\chi\,d\xi.
\end{aligned}
\]
The remaining nonlocal term becomes
\[
\int_{\mathbb R^{4}}
\left(\int_{\mathbb R^{4}}
\frac{e^{U(y)}}{|x-y|^{\alpha}}\,dy\right)
e^{U(x)}\phi(x)S^{*}\chi(x)\,dx
=
\kappa_{\alpha}\mu_{0}(\alpha)
\int_{\mathbb S^{4}}\Phi\chi\,d\xi.
\]
All these integrals are finite. This follows from Cauchy--Schwarz on the
sphere and the boundedness of \(\mathcal R_{\alpha}\) on \(L^{2}\). Therefore
\[
\langle\mathcal L_{\alpha}\Phi,\chi\rangle
=
\langle\mathscr L_{U}\phi,S^{*}\chi\rangle=0
\]
by the distributional equation \eqref{eq:linearized}. This proves
\eqref{eq:sphere-away-pole}.
\end{proof}

We next identify the distributional defects compatible with the weighted \(L^{2}\) class.

\begin{lemma}\label{lem:pole-defect}
Let \(p\in\mathbb S^{4}\), \(\Phi\in L^{2}(\mathbb S^{4})\), and assume that
\[
\mathcal L_{\alpha}\Phi=0
\qquad\hbox{in }\mathcal D'(\mathbb S^{4}\setminus\{p\}).
\]
Then, in geodesic coordinates \((z_{1},\ldots,z_{4})\) centered at \(p\),
\begin{equation}\label{eq:defect-form}
\mathcal L_{\alpha}\Phi
=c_{0}\delta_{p}+\sum_{j=1}^{4}c_{j}\partial_{z_{j}}\delta_{p}
\end{equation}
for suitable constants \(c_{0},\ldots,c_{4}\).
\end{lemma}

\begin{proof}
Set \(T:=\mathcal L_{\alpha}\Phi\). The distribution \(T\) is supported at
\(p\), hence in local coordinates it has the form
\begin{equation}\label{eq:point-supported-general}
T=\sum_{|\beta|\leq m}c_{\beta}\partial^{\beta}\delta_{0}
\end{equation}
for some finite \(m\). Since \(\Phi\in L^{2}\),
\(P_{4}^{\mathbb S^{4}}\Phi\in H^{-4}\). Moreover,
\(\mathcal R_{\alpha}\Phi+\mu_{0}(\alpha)\Phi\in L^{2}\subset H^{-4}\).
Thus
\begin{equation}\label{eq:defect-Hminus4}
T\in H^{-4}_{\operatorname{loc}}.
\end{equation}
Choose a smooth cut-off which equals one near the origin. Since \(T\) is
supported at the origin, multiplication by this cut-off leaves \(T\)
unchanged, so the usual Fourier characterization of \(H^{-4}\) applies. We
show that \eqref{eq:defect-Hminus4} excludes every derivative of order at
least two. Let \(m\) be the largest order occurring in
\eqref{eq:point-supported-general}, and let
\(P_{m}(\xi)=\sum_{|\beta|=m}c_{\beta}(i\xi)^{\beta}\). If the top-order
coefficients are not all zero, then \(P_m\) is a nonzero homogeneous
polynomial. There are a measurable set \(E\subset\mathbb S^3\) of positive
measure and \(c>0\) such that \(|P_m(\omega)|\geq c\) on \(E\). The Fourier
transform of \eqref{eq:point-supported-general} is \(P_m(\xi)\) plus a
polynomial of degree at most \(m-1\). The lower-order polynomial is bounded by
\(Cr^{m-1}\) on \(r\mathbb S^3\). Hence, uniformly for \(\omega\in E\) and all
sufficiently large \(r\), the absolute value of the full Fourier transform is
at least \(cr^m/2\). Its \(H^{-4}\)-norm at high frequency therefore contains
a fixed positive multiple of
\[
\int_{1}^{\infty}r^{3+2m}(1+r^{2})^{-4}\,dr.
\]
This integral is finite exactly when \(m<2\). Consequently \(m\leq1\), which
is precisely \eqref{eq:defect-form}.
\end{proof}

The first spherical harmonics remove all of these possible defects.

\begin{proposition}\label{prop:global-sphere-equation}
Let \(\phi\in L^{2}_{w}(\mathbb R^{4})\) be a distributional solution of
\eqref{eq:linearized}, and put \(\Phi=S_{*}\phi\). Then
\(\Phi\in H^{4}(\mathbb S^{4})\) and
\begin{equation*}
\mathcal L_{\alpha}\Phi=0
\qquad\hbox{in }\mathbb S^{4}.
\end{equation*}
\end{proposition}

\begin{proof}
Lemma~\ref{lem:weak-conformal-transfer} gives \(\Phi\in L^{2}(\mathbb S^{4})\)
and \eqref{eq:sphere-away-pole}. Lemma~\ref{lem:pole-defect} then gives
\[
T:=\mathcal L_{\alpha}\Phi
=c_{0}\delta_{p}+\sum_{j=1}^{4}c_{j}\partial_{z_{j}}\delta_{p}.
\]
For every \(Y\in\mathcal H_{1}\), equations \eqref{eq:Paneitz-eigen},
\eqref{eq:mu0-mu1}, \eqref{eq:C-alpha} and \eqref{eq:kappa-alpha} give
\[
\mathcal L_{\alpha}Y
=\left[24-\kappa_{\alpha}
\bigl(\mu_{1}(\alpha)+\mu_{0}(\alpha)\bigr)\right]Y=0.
\]
Since \(\Phi\in L^{2}\), distributional symmetry is legitimate and
\begin{equation}\label{eq:defect-pairing}
\langle T,Y\rangle
=\langle\Phi,\mathcal L_{\alpha}Y\rangle=0.
\end{equation}
Take \(Y=\xi_{5}\). At \(p=(0,0,0,0,-1)\), one has
\(\xi_{5}(p)=-1\) and \(\nabla_{\mathbb S^{4}}\xi_{5}(p)=0\), so
\eqref{eq:defect-pairing} gives \(c_{0}=0\). Next take
\(Y=\xi_{j}\), \(1\leq j\leq4\). These functions vanish at \(p\), while
\(\nabla_{\mathbb S^{4}}\xi_{1}(p),\ldots,
\nabla_{\mathbb S^{4}}\xi_{4}(p)\) form a basis of
\(T_{p}\mathbb S^{4}\). Since
\(\langle\partial_{z_j}\delta_p,Y\rangle=-\partial_{z_j}Y(p)\), the resulting
linear system is nonsingular. Thus \(c_{1}=\cdots=c_{4}=0\), and \(T=0\).
Finally,
\[
P_{4}^{\mathbb S^{4}}\Phi
=\kappa_{\alpha}
\bigl(\mathcal R_{\alpha}\Phi+\mu_{0}(\alpha)\Phi\bigr)
\in L^{2}(\mathbb S^{4}).
\]
Elliptic regularity on the compact sphere gives
\(\Phi\in H^{4}(\mathbb S^{4})\).
\end{proof}

\section{Complete spectrum, Morse index and coercivity}

\begin{theorem}\label{thm:spectrum-proof}
The spectrum of \(\mathcal L_{\alpha}\) is
\begin{equation*}
\operatorname{spec}(\mathcal L_{\alpha})
=\{\Lambda_{k}(\alpha):k=0,1,2,\ldots\},
\end{equation*}
where \(\Lambda_{k}(\alpha)\) is given by \eqref{eq:Lambda-k-intro}, with multiplicity \(\dim\mathcal H_{k}\). Furthermore,
\begin{equation*}
\Lambda_{0}(\alpha)=-6(8-\alpha),
\end{equation*}
\begin{equation}\label{eq:Lambda1}
\Lambda_{1}(\alpha)=0,
\end{equation}
and
\begin{equation}\label{eq:Lambda-high}
\Lambda_{k}(\alpha)
\geq\frac45 k(k+1)(k+2)(k+3),
\qquad k\geq2.
\end{equation}
\end{theorem}

\begin{proof}
Proposition~\ref{prop:Riesz-operator} and \eqref{eq:Paneitz-eigen} show that \(P_{4}^{\mathbb S^{4}}\) and \(\mathcal R_{\alpha}\) are simultaneously diagonal in the spherical harmonic decomposition. Hence every \(\mathcal H_{k}\) is an eigenspace with eigenvalue \eqref{eq:Lambda-k-intro}. Since \(\mathcal R_{\alpha}\) is bounded and self-adjoint, \(\mathcal L_{\alpha}\) is self-adjoint on the domain of \(P_{4}^{\mathbb S^{4}}\), namely \(H^{4}(\mathbb S^{4})\), and has compact resolvent. Moreover, if \(\lambda_k=k(k+1)(k+2)(k+3)\), then
\[
\Lambda_{k+1}(\alpha)-\Lambda_k(\alpha)
=(\lambda_{k+1}-\lambda_k)
+\kappa_{\alpha}(\mu_k(\alpha)-\mu_{k+1}(\alpha))>0.
\]
Thus the sequence \(\Lambda_k(\alpha)\) is strictly increasing, and the
multiplicity of \(\Lambda_k(\alpha)\) is exactly \(\dim\mathcal H_k\).

For \(k=0\), formulas \eqref{eq:C-alpha}, \eqref{eq:kappa-alpha} and \eqref{eq:mu0-mu1} give
\[
\kappa_{\alpha}\mu_{0}(\alpha)=3(8-\alpha),
\]
and hence \(\Lambda_{0}=-2\kappa_{\alpha}\mu_{0}=-6(8-\alpha)\).
For \(k=1\),
\[
\mu_{1}(\alpha)+\mu_{0}(\alpha)
=\frac{2^{9-\alpha}\pi^{2}}
{(8-\alpha)(6-\alpha)(4-\alpha)},
\]
so
\[
\kappa_{\alpha}
\bigl(\mu_{1}(\alpha)+\mu_{0}(\alpha)\bigr)=24.
\]
Since the Paneitz eigenvalue on \(\mathcal H_{1}\) is also \(24\), this proves \eqref{eq:Lambda1}.

Let \(k\geq2\). By the strict monotonicity \eqref{eq:mu-monotone},
\[
\kappa_{\alpha}
\bigl(\mu_{k}(\alpha)+\mu_{0}(\alpha)\bigr)
<\kappa_{\alpha}
\bigl(\mu_{1}(\alpha)+\mu_{0}(\alpha)\bigr)=24.
\]
Since \(k(k+1)(k+2)(k+3)\geq120\),
\[
\kappa_{\alpha}
\bigl(\mu_{k}(\alpha)+\mu_{0}(\alpha)\bigr)
\leq\frac15 k(k+1)(k+2)(k+3).
\]
This proves \eqref{eq:Lambda-high}.
\end{proof}

\begin{proof}[Proof of Theorem~\ref{thm:full-spectrum}]
The spectral formula and the signs of the eigenvalues follow from Theorem~\ref{thm:spectrum-proof}. Thus \(\mathcal H_{0}\) is the unique negative eigenspace, \(\mathcal H_{1}\) is the kernel, and all higher modes are positive. This proves Morse index one and nullity five.

Let
\[
\Phi=\sum_{k=2}^{\infty}\Phi_{k},
\qquad \Phi_{k}\in\mathcal H_{k}.
\]
Then Theorem~\ref{thm:spectrum-proof} gives
\[
\begin{aligned}
\langle\mathcal L_{\alpha}\Phi,\Phi\rangle
=\sum_{k=2}^{\infty}
\Lambda_{k}(\alpha)\|\Phi_{k}\|_{2}^{2}\geq\frac45
\sum_{k=2}^{\infty}
k(k+1)(k+2)(k+3)\|\Phi_{k}\|_{2}^{2}=\frac45
\langle P_{4}^{\mathbb S^{4}}\Phi,\Phi\rangle.
\end{aligned}
\]
This is \eqref{eq:coercivity-intro}.

On \(\mathcal H_{1}^{\perp}\), zero is separated from the spectrum:
\(|\Lambda_{0}(\alpha)|\geq24\), and
\(\Lambda_{k}(\alpha)\geq96\) for \(k\geq2\). Moreover,
\(\mu_k(\alpha)\leq\mu_0(\alpha)\) and
\(\kappa_\alpha\mu_0(\alpha)=3(8-\alpha)\), so
\[
0\leq\kappa_\alpha\bigl(\mu_k(\alpha)+\mu_0(\alpha)\bigr)
\leq48
\qquad\hbox{for all }k\geq0,\ 0<\alpha<4.
\]
Consequently, for \(k\geq2\), \(\Lambda_k(\alpha)\) is bounded above and below
by fixed positive multiples of
\(k(k+1)(k+2)(k+3)\), independently of \(\alpha\). Division by the
eigenvalues in the spherical harmonic expansion therefore gives
\[
\|\Phi\|_{H^{4}(\mathbb S^{4})}
\leq C\|\mathcal L_{\alpha}\Phi\|_{L^{2}(\mathbb S^{4})},
\qquad \Phi\perp_{L^{2}}\mathcal H_{1},
\]
with \(C\) independent of \(0<\alpha<4\), after fixing an equivalent spectral
norm on \(H^{4}(\mathbb S^{4})\).
\end{proof}

\begin{proof}[Proof of Corollary~\ref{cor:euclidean-nondegeneracy}]
Let \(\Phi=S_{*}\phi\). Proposition~\ref{prop:global-sphere-equation} gives \(\Phi\in H^{4}(\mathbb S^{4})\) and \(\mathcal L_{\alpha}\Phi=0\). Theorem~\ref{thm:full-spectrum} yields
\[
\Phi=\sum_{j=1}^{5}b_{j}\xi_{j}.
\]
From \eqref{eq:stereo} and \eqref{eq:kernel-functions}--\eqref{eq:kernel-scale},
\[
S_{*}\phi_{j}=\frac{8-\alpha}{2}\xi_{j},
\qquad 1\leq j\leq5.
\]
Therefore \(\phi=\sum_{j=1}^{5}a_{j}\phi_{j}\). In particular
\(\phi\) is smooth and bounded and belongs to the weighted energy class
\[
H^{2}_{w}(\mathbb R^{4})
:=\left\{\phi\in L^{2}_{w}(\mathbb R^{4}):\Delta\phi\in L^{2}(\mathbb R^{4})\right\}.
\]
Translation and dilation give the statement at \(U_{\mu,\zeta}\).

The Morse index statement is made on the form domain
\[
S^{*}H^{2}(\mathbb S^{4})
:=\{\phi=S^{*}\Phi:\Phi\in H^{2}(\mathbb S^{4})\}.
\]
On this domain, Proposition~\ref{prop:conformal-linearized} and
Theorem~\ref{thm:full-spectrum} show that \(\mathcal Q_U\) has one negative
direction, corresponding to constants on \(\mathbb S^{4}\), and kernel
generated by the five symmetry modes. For every \(\phi=S^{*}\Phi\) with
\(\Phi\perp_{L^{2}}\mathcal H_{0}\oplus\mathcal H_{1}\), one has
\begin{equation}\label{eq:Euclidean-coercivity}
\mathcal Q_{U}(\phi)
\geq\frac45\int_{\mathbb R^{4}}|\Delta\phi|^{2}\,dx.
\end{equation}
Equivalently, the orthogonality conditions may be written as
\[
\int_{\mathbb R^{4}}\phi(x)\omega(x)^{4}\,dx=0,
\qquad
\int_{\mathbb R^{4}}\phi(x)\phi_{j}(x)\omega(x)^{4}\,dx=0,
\quad 1\leq j\leq5.
\]
\end{proof}

\begin{corollary}\label{cor:nondegenerate-all-bubbles}
For every \(\mu>0\) and \(\zeta\in\mathbb R^{4}\), the kernel of the linearized operator at \(U_{\mu,\zeta}\) in
\[
L^{2}_{w,\mu,\zeta}(\mathbb R^{4})
:=\left\{\phi:
\int_{\mathbb R^{4}}
\frac{|\phi(x)|^{2}\mu^{4}}
{(1+\mu^{2}|x-\zeta|^{2})^{4}}\,dx<+\infty
\right\}
\]
is spanned by
\[
\frac{\partial U_{\mu,\zeta}}{\partial\zeta_{1}},\ldots,
\frac{\partial U_{\mu,\zeta}}{\partial\zeta_{4}},
\frac{\partial U_{\mu,\zeta}}{\partial\mu}.
\]
On the form domain obtained from \(S^{*}H^{2}(\mathbb S^{4})\) by the same
translation and dilation, the quadratic form has Morse index one. The
coercivity estimate obtained from \eqref{eq:Euclidean-coercivity} is invariant
under these transformations.
\end{corollary}

\begin{proof}
Apply Corollary~\ref{cor:euclidean-nondegeneracy} after the change of variables \(x\mapsto\mu(x-\zeta)\). The weighted space, the kernel directions, the quadratic form and the orthogonality conditions are transported by the conformal symmetry \eqref{eq:invariance}.
\end{proof}

\section{Sharp Adams--Choquard inequality and local stability}\label{sec:stability}

We prove Corollary~\ref{cor:sharp-Adams-Choquard} and Theorem~\ref{thm:Adams-Choquard-stability}. The inequality itself follows from two sharp endpoint inequalities. The stability theorem uses the complete spectrum and a nearest-point modulation argument. Throughout this section, \(q_{\alpha}\) is defined by \eqref{eq:q-alpha}.

\subsection{The sharp inequality and equality cases}

We first derive the endpoint inequality and characterize all equality cases.

\begin{proposition}\label{prop:sharp-Adams-Choquard}
For every \(\Phi\in H^{2}(\mathbb S^{4})\), inequality \eqref{eq:sharp-Adams-Choquard} holds. Equivalently,
\begin{equation}\label{eq:deficit-nonnegative}
\mathcal D_{\alpha}(\Phi)\geq0.
\end{equation}
Equality in \eqref{eq:deficit-nonnegative} holds if and only if \(\Phi\in\mathcal M_{\alpha}\).
\end{proposition}

\begin{proof}
Apply Lemma~\ref{lem:spherical-HLS} with \(f=e^{\Phi}\). We obtain
\begin{equation}\label{eq:HLS-log-step}
\log\frac{\mathcal I_{\alpha}(\Phi)}{|\mathbb S^{4}|\mu_{0}(\alpha)}
\leq\frac{2}{q_{\alpha}}
\log\left(\frac{1}{|\mathbb S^{4}|}
\int_{\mathbb S^{4}}e^{q_{\alpha}\Phi}\,d\xi\right).
\end{equation}
Separate the mean:
\begin{equation*}
\begin{aligned}
\frac{2}{q_{\alpha}}
\log\left(\frac{1}{|\mathbb S^{4}|}
\int_{\mathbb S^{4}}e^{q_{\alpha}\Phi}\,d\xi\right)
={}&2\overline\Phi+\frac{2}{q_{\alpha}}
\log\left(\frac{1}{|\mathbb S^{4}|}
\int_{\mathbb S^{4}}e^{q_{\alpha}(\Phi-\overline\Phi)}\,d\xi\right).
\end{aligned}
\end{equation*}
Use Lemma~\ref{lem:Beckner-Adams} with
\begin{equation*}
W:=\frac{q_{\alpha}}{4}(\Phi-\overline\Phi).
\end{equation*}
Since the Paneitz operator annihilates constants,
\begin{equation}\label{eq:Adams-step}
\log\left(\frac{1}{|\mathbb S^{4}|}
\int_{\mathbb S^{4}}e^{q_{\alpha}(\Phi-\overline\Phi)}\,d\xi\right)
\leq\frac{q_{\alpha}^{2}}{128\pi^{2}}
\int_{\mathbb S^{4}}\Phi P_{4}^{\mathbb S^{4}}\Phi\,d\xi.
\end{equation}
The identity
\begin{equation*}
\frac{2}{q_{\alpha}}\frac{q_{\alpha}^{2}}{128\pi^{2}}
=\frac{1}{8\pi^{2}(8-\alpha)}
\end{equation*}
proves \eqref{eq:sharp-Adams-Choquard}.

It remains to identify equality. Equality in \eqref{eq:HLS-log-step} requires
\begin{equation*}
e^{\Phi}=c_{1}J_{a}^{1/q_{\alpha}}.
\end{equation*}
Equality in \eqref{eq:Adams-step} requires
\begin{equation*}
\frac{q_{\alpha}}{4}(\Phi-\overline\Phi)
=c_{2}+\frac14\log J_{a}.
\end{equation*}
Since \(q_{\alpha}^{-1}=(8-\alpha)/8\), each condition places
\(\Phi\) in the family
\(c+(8-\alpha)\log J_a/8\). If the parameters obtained from the two
equality statements are initially different, then the corresponding
Jacobians differ by a positive constant. Their normalization
\(\int_{\mathbb S^{4}}J_a\,d\xi=|\mathbb S^{4}|\) forces that constant
to be one, and hence the parameters agree. Conversely, every function of this
form is an equality case in both sharp inequalities.
\end{proof}

The functional \(\mathcal D_{\alpha}\) is invariant under addition of constants. It is also invariant under the conformal action
\begin{equation*}
(T_{\tau}\Phi)(\xi)
:=\Phi(\tau(\xi))+\frac{8-\alpha}{8}\log J_{\tau}(\xi),
\end{equation*}
where \(\tau\) is a conformal diffeomorphism of \(\mathbb S^{4}\) and \(J_{\tau}\) is its Jacobian. Indeed, the sharp HLS quotient is invariant by \eqref{eq:HLS-weight-cancellation}. The Beckner functional is invariant under the corresponding action on \(W=q_{\alpha}\Phi/4\). Thus
\begin{equation}\label{eq:deficit-invariance}
\mathcal D_{\alpha}(T_{\tau}\Phi)=\mathcal D_{\alpha}(\Phi).
\end{equation}
The Paneitz seminorm of a difference is invariant under composition with \(\tau\):
\begin{equation}\label{eq:Paneitz-invariance}
\int_{\mathbb S^{4}}((\Phi-\Psi)\circ\tau)
P_{4}^{\mathbb S^{4}}((\Phi-\Psi)\circ\tau)\,d\xi
=\int_{\mathbb S^{4}}(\Phi-\Psi)
P_{4}^{\mathbb S^{4}}(\Phi-\Psi)\,d\xi.
\end{equation}

\subsection{Quadratic expansion of the deficit}

For \(f,g\in L^{q_{\alpha}}(\mathbb S^{4})\), set
\begin{equation*}
\mathcal B_{\alpha}(f,g)
:=\int_{\mathbb S^{4}}\int_{\mathbb S^{4}}
\frac{f(\xi)g(\eta)}{|\xi-\eta|^{\alpha}}\,d\xi d\eta.
\end{equation*}
Lemma~\ref{lem:spherical-HLS} and polarization give
\begin{equation}\label{eq:B-alpha-bound}
|\mathcal B_{\alpha}(f,g)|
\leq C(\alpha)\|f\|_{L^{q_{\alpha}}(\mathbb S^{4})}
\|g\|_{L^{q_{\alpha}}(\mathbb S^{4})}.
\end{equation}

We next expand the deficit near the normalized extremal.

\begin{lemma}\label{lem:deficit-expansion}
Let \(K\Subset(0,4)\). There exist \(\rho_{K}>0\) and \(C_{K}>0\) such that,
for every \(\alpha\in K\) and every \(\Psi\in H^{2}(\mathbb S^{4})\) satisfying
\begin{equation*}
\Psi\perp_{L^{2}}\mathcal H_{0},
\qquad
\|\Psi\|_{H^{2}(\mathbb S^{4})}\leq\rho_{K},
\end{equation*}
one has
\begin{equation}\label{eq:deficit-expansion}
\left|\mathcal D_{\alpha}(\Psi)
-\frac12\langle\mathcal L_{\alpha}\Psi,\Psi\rangle_{L^{2}(\mathbb S^{4})}\right|
\leq C_{K}\|\Psi\|_{H^{2}(\mathbb S^{4})}^{3}.
\end{equation}
\end{lemma}

\begin{proof}
Let \(q_{+}:=\max_{\alpha\in K}q_{\alpha}\). Fix \(r>q_{+}\). Apply Lemma
\ref{lem:Beckner-Adams} to \(W=\pm r\Psi/4\). Since
\(e^{r|\Psi|}\leq e^{r\Psi}+e^{-r\Psi}\), there is \(C=C(K,\rho_K)\) such that
\begin{equation}\label{eq:uniform-exponential}
\sup_{\substack{\Psi\perp_{L^{2}}\mathcal H_{0}\\
\|\Psi\|_{H^{2}}\leq\rho_K}}
\|e^{|\Psi|}\|_{L^{r}(\mathbb S^{4})}\leq C.
\end{equation}
The embedding \(H^{2}(\mathbb S^{4})\hookrightarrow L^{s}(\mathbb S^{4})\)
holds for every finite \(s\). Define
\begin{equation*}
R_{\Psi}:=e^{\Psi}-1-\Psi-\frac12\Psi^{2}.
\end{equation*}
The pointwise estimate
\[
|R_{\Psi}|\leq C|\Psi|^{3}e^{|\Psi|}
\]
together with H\"older's inequality and \eqref{eq:uniform-exponential} gives,
uniformly for \(\alpha\in K\), the following bound. Indeed, because
\(r>q_\alpha\), one may choose a finite exponent \(s_\alpha\), uniformly
bounded for \(\alpha\in K\), such that
\(q_\alpha^{-1}=r^{-1}+s_\alpha^{-1}\), and then use the embedding into
\(L^{3s_\alpha}\):
\begin{equation}\label{eq:exp-remainder}
\|R_{\Psi}\|_{L^{q_{\alpha}}(\mathbb S^{4})}
\leq C_{K}\|\Psi\|_{H^{2}(\mathbb S^{4})}^{3}.
\end{equation}

We now expand every term in the nonlocal energy. By symmetry of
\(\mathcal B_{\alpha}\),
\begin{equation}\label{eq:full-B-expansion}
\begin{aligned}
\mathcal I_{\alpha}(\Psi)
={}&\mathcal B_{\alpha}(1,1)
+2\mathcal B_{\alpha}(1,\Psi)
+\mathcal B_{\alpha}(\Psi,\Psi)
+\mathcal B_{\alpha}(1,\Psi^{2})\\
&+2\mathcal B_{\alpha}(1,R_{\Psi})
+\mathcal B_{\alpha}(\Psi,\Psi^{2})
+2\mathcal B_{\alpha}(\Psi,R_{\Psi})\\
&+\frac14\mathcal B_{\alpha}(\Psi^{2},\Psi^{2})
+\mathcal B_{\alpha}(\Psi^{2},R_{\Psi})
+\mathcal B_{\alpha}(R_{\Psi},R_{\Psi}).
\end{aligned}
\end{equation}
Since \(\mathcal R_{\alpha}1=\mu_{0}(\alpha)\),
\[
\mathcal B_{\alpha}(1,1)=|\mathbb S^{4}|\mu_{0}(\alpha),
\qquad
\mathcal B_{\alpha}(1,\Psi)=0,
\qquad
\mathcal B_{\alpha}(1,\Psi^{2})=\mu_{0}(\alpha)\|\Psi\|_{2}^{2}.
\]
The remaining six terms in \eqref{eq:full-B-expansion} are estimated separately
from \eqref{eq:B-alpha-bound}, \eqref{eq:exp-remainder} and the Sobolev
embeddings:
\begin{equation}\label{eq:remainder-B-estimates}
\begin{aligned}
|\mathcal B_{\alpha}(1,R_{\Psi})|
&\leq C_K\|\Psi\|_{H^{2}}^{3},\\
|\mathcal B_{\alpha}(\Psi,\Psi^{2})|
&\leq C_K\|\Psi\|_{H^{2}}^{3},\\
|\mathcal B_{\alpha}(\Psi,R_{\Psi})|
&\leq C_K\|\Psi\|_{H^{2}}^{4},\\
|\mathcal B_{\alpha}(\Psi^{2},\Psi^{2})|
&\leq C_K\|\Psi\|_{H^{2}}^{4},\\
|\mathcal B_{\alpha}(\Psi^{2},R_{\Psi})|
&\leq C_K\|\Psi\|_{H^{2}}^{5},\\
|\mathcal B_{\alpha}(R_{\Psi},R_{\Psi})|
&\leq C_K\|\Psi\|_{H^{2}}^{6}.
\end{aligned}
\end{equation}
After decreasing \(\rho_K\) if necessary, \eqref{eq:full-B-expansion} and
\eqref{eq:remainder-B-estimates} yield
\begin{equation}\label{eq:I-expansion}
\mathcal I_{\alpha}(\Psi)
=|\mathbb S^{4}|\mu_{0}(\alpha)
+\mu_{0}(\alpha)\|\Psi\|_{2}^{2}
+\langle\mathcal R_{\alpha}\Psi,\Psi\rangle
+\mathcal E_{\alpha}(\Psi),
\end{equation}
where
\begin{equation*}
|\mathcal E_{\alpha}(\Psi)|
\leq C_K\|\Psi\|_{H^{2}}^{3}.
\end{equation*}
The quantities \(\mu_0(\alpha)\) are bounded above and away from zero on
\(K\). Hence the perturbation of
\(|\mathbb S^{4}|\mu_0(\alpha)\) in \eqref{eq:I-expansion} is
\(O_K(\|\Psi\|_{H^{2}}^{2})\). Using
\(\log(1+t)=t+O(t^{2})\), we obtain
\begin{equation}\label{eq:log-I-expansion}
\log\frac{\mathcal I_{\alpha}(\Psi)}
{|\mathbb S^{4}|\mu_{0}(\alpha)}
=
\frac{\mu_{0}(\alpha)\|\Psi\|_{2}^{2}
+\langle\mathcal R_{\alpha}\Psi,\Psi\rangle}
{|\mathbb S^{4}|\mu_{0}(\alpha)}
+O_K(\|\Psi\|_{H^{2}}^{3}).
\end{equation}
Finally, substitute \eqref{eq:log-I-expansion} into \eqref{eq:deficit} and use
\eqref{eq:kappa-identities}. The quadratic term is
\(\frac12\langle\mathcal L_{\alpha}\Psi,\Psi\rangle\), and
\eqref{eq:deficit-expansion} follows.
\end{proof}

The spectral gap now gives a coercive estimate in the fixed gauge.

\begin{proposition}\label{prop:gauge-fixed}
There exists \(\rho_{\alpha}>0\) such that every \(\Psi\in H^{2}(\mathbb S^{4})\) satisfying
\begin{equation}\label{eq:gauge-smallness}
\Psi\perp_{L^{2}}\mathcal H_{0}\oplus\mathcal H_{1},
\qquad
\|\Psi\|_{H^{2}(\mathbb S^{4})}\leq\rho_{\alpha},
\end{equation}
obeys
\begin{equation}\label{eq:gauge-lower}
\mathcal D_{\alpha}(\Psi)
\geq\frac14\int_{\mathbb S^{4}}\Psi P_{4}^{\mathbb S^{4}}\Psi\,d\xi.
\end{equation}
There is also \(C=C(\alpha)>0\) such that
\begin{equation*}
\mathcal D_{\alpha}(\Psi)
\leq C\int_{\mathbb S^{4}}\Psi P_{4}^{\mathbb S^{4}}\Psi\,d\xi.
\end{equation*}
\end{proposition}

\begin{proof}
By Theorem~\ref{thm:full-spectrum},
\begin{equation*}
\langle\mathcal L_{\alpha}\Psi,\Psi\rangle
\geq\frac45\int_{\mathbb S^{4}}\Psi P_{4}^{\mathbb S^{4}}\Psi\,d\xi.
\end{equation*}
On the subspace in \eqref{eq:gauge-smallness}, the Paneitz energy is
equivalent to the squared \(H^{2}\)-norm. Applying
Lemma~\ref{lem:deficit-expansion}, we obtain
\[
\mathcal D_{\alpha}(\Psi)
\geq\frac25\int_{\mathbb S^{4}}\Psi P_{4}^{\mathbb S^{4}}\Psi\,d\xi
-C\|\Psi\|_{H^{2}}^{3}.
\]
Decrease \(\rho_{\alpha}\) so that the last term is at most \(3/20\) of the Paneitz energy. This proves \eqref{eq:gauge-lower}. The upper estimate follows from the same expansion and the boundedness of \(\mathcal R_{\alpha}\) on \(L^{2}(\mathbb S^{4})\).
\end{proof}

\subsection{Modulation and the sharp local constant}

To pass from the fixed gauge to the full extremal manifold, we use a nearest-point modulation argument.

\begin{lemma}\label{lem:local-modulation}
There exists \(\varepsilon_{\alpha}>0\) with the following property. If
\(d_{P}(\Phi,\mathcal M_{\alpha})<\varepsilon_{\alpha}\), then the nearest
extremal is uniquely determined in \(\mathcal M_{\alpha}/\mathbb R\). Moreover,
there are an extremal \(\Theta\in\mathcal M_{\alpha}\), a conformal map
\(\tau\), a constant \(c\), and \(\Psi\in H^{2}(\mathbb S^{4})\) such that
\begin{equation}\label{eq:modulation-decomposition}
\Psi=T_{\tau}\Phi-c,
\qquad
\Psi\perp_{L^{2}}\mathcal H_{0}\oplus\mathcal H_{1},
\end{equation}
\begin{equation}\label{eq:distance-exact}
d_{P}(\Phi,\mathcal M_{\alpha})^{2}
=\int_{\mathbb S^{4}}\Psi P_{4}^{\mathbb S^{4}}\Psi\,d\xi,
\end{equation}
and
\begin{equation}\label{eq:deficit-exact-modulation}
\mathcal D_{\alpha}(\Phi)=\mathcal D_{\alpha}(\Psi).
\end{equation}
In addition,
\begin{equation}\label{eq:modulation-H2-control}
\|\Psi\|_{H^{2}(\mathbb S^{4})}
\leq C(\alpha)d_{P}(\Phi,\mathcal M_{\alpha}).
\end{equation}
The constants can be chosen uniformly for \(\alpha\) in a compact subset of
\((0,4)\).
\end{lemma}

\begin{proof}
We work first near the normalized extremal. Set
\[
m_{\alpha}(a):=\frac{8-\alpha}{8}\log J_{a},
\qquad
\widehat m_{\alpha}(a):=m_{\alpha}(a)-\overline{m_{\alpha}(a)}.
\]
The mean-zero space \(\mathcal H_{0}^{\perp}\cap H^{2}(\mathbb S^{4})\),
endowed with
\begin{equation*}
(f,g)_{P}:=\int_{\mathbb S^{4}}fP_{4}^{\mathbb S^{4}}g\,d\xi,
\end{equation*}
is a Hilbert space with norm equivalent to the \(H^{2}\)-norm. The map
\(a\mapsto\widehat m_{\alpha}(a)\) is smooth and
\begin{equation}\label{eq:extremal-derivative}
\partial_{a_{j}}\widehat m_{\alpha}(0)=(8-\alpha)\xi_{j},
\qquad 1\leq j\leq5.
\end{equation}
Thus its differential has rank five.

For mean-zero \(f\) and \(a\) near zero, define
\begin{equation}\label{eq:normal-equations-map}
G_{j}(a,f):=
\bigl(f-\widehat m_{\alpha}(a),
\partial_{a_j}\widehat m_{\alpha}(a)\bigr)_{P},
\qquad 1\leq j\leq5.
\end{equation}
At \((a,f)=(0,0)\),
\begin{equation}\label{eq:normal-equations-derivative}
\partial_{a_k}G_j(0,0)
=-(8-\alpha)^{2}
\int_{\mathbb S^{4}}\xi_kP_{4}^{\mathbb S^{4}}\xi_j\,d\xi
=-24(8-\alpha)^{2}\frac{|\mathbb S^{4}|}{5}\delta_{jk}.
\end{equation}
The implicit function theorem therefore gives a unique smooth parameter
\(a=a(f)\) satisfying \(G(a(f),f)=0\) for \(f\) small. The Hessian of
\(a\mapsto\|f-\widehat m_{\alpha}(a)\|_{P}^{2}\) equals
\(-2D_aG(a,f)\). At \((0,0)\) it is twice the positive Gram matrix of
\((8-\alpha)\xi_1,\ldots,(8-\alpha)\xi_5\); for small \(f\), the
remaining terms tend to zero in operator norm. The Hessian is therefore
positive definite after the neighbourhood is decreased. Hence \(\widehat m_{\alpha}(a(f))\) is the unique
local nearest point.

This nearest point is also the global nearest point when \(f\) is sufficiently
small. Indeed, the orbit is proper in the Paneitz norm. To see this, put
\(W_a=\frac14\log J_a\). Since \(\int_{\mathbb S^{4}}J_a\,d\xi=|\mathbb S^{4}|\),
the equality identity in Lemma~\ref{lem:Beckner-Adams} gives
\begin{equation}\label{eq:proper-extremal-orbit}
\int_{\mathbb S^{4}}W_aP_{4}^{\mathbb S^{4}}W_a\,d\xi
=-32\pi^{2}\overline{W_a}.
\end{equation}
The explicit formula for \(J_a\) gives
\[
W_a=\log(1-|a|^{2})-\log(1-2a\cdot\xi+|a|^{2}).
\]
We justify the boundedness of the mean of the second term. Write
\(a=r\theta\), where \(r=|a|\) and \(\theta\in\mathbb S^{4}\). For
\(r\geq1/2\),
\[
1-2a\cdot\xi+|a|^{2}
=(1-r)^{2}+2r(1-\theta\cdot\xi)
\asymp (1-r)^{2}+d_{\mathbb S^{4}}(\xi,\theta)^{2}
\]
near \(\theta\), with constants independent of \(r\) and \(\theta\). Hence
\[
\int_{0}^{1}\rho^{3}
\left|\log\bigl((1-r)^{2}+\rho^{2}\bigr)\right|\,d\rho\leq C.
\]
Away from \(\theta\), the logarithm is uniformly bounded. Therefore
\[
\int_{\mathbb S^{4}}
\left|\log(1-2a\cdot\xi+|a|^{2})\right|\,d\xi\leq C
\]
uniformly as \(r\uparrow1\). Since \(\log(1-r^{2})\to-\infty\), it follows
that \(\overline{W_a}\to-\infty\). Thus the right-hand side of
\eqref{eq:proper-extremal-orbit} tends to \(+\infty\). On compact subsets of
\(B^{5}\), the parametrization is injective modulo constants. Indeed, if
\(\widehat m_\alpha(a)=\widehat m_\alpha(b)\), then \(J_a\) and \(J_b\) differ
by a positive constant; their common normalization
\(\int_{\mathbb S^4}J_a=\int_{\mathbb S^4}J_b=|\mathbb S^4|\) forces this
constant to be one, and comparison of the affine denominators gives \(a=b\).
It follows by compactness that
\(\widehat m_{\alpha}(a_n)\to0\) in \(H^{2}\) only if \(a_n\to0\).
Consequently, for every sufficiently small parameter neighbourhood
\(U\) of zero, there is \(\eta>0\) such that
\[
\|\widehat m_\alpha(a)\|_P\geq4\eta
\qquad\hbox{for }a\in B^5\setminus U.
\]
If \(\|f\|_P<\eta\), then the distance from \(f\) to every orbit point
with parameter outside \(U\) is at least \(3\eta\), whereas its distance
to the normalized extremal is smaller than \(\eta\). Hence the local
critical point constructed above is the unique global nearest point. This
proves the asserted local uniqueness in the quotient by constants.

Now let \(\Phi\) satisfy the hypothesis of the lemma. If
\(d_P(\Phi,\mathcal M_\alpha)=0\), then \([\Phi]\) already belongs to
\(\mathcal M_\alpha/\mathbb R\), and all conclusions follow with
\(\Psi=0\). Assume henceforth that the distance is positive. Choose an
extremal \(\Theta_{0}\) whose distance to \(\Phi\) is smaller than
\(2d_P(\Phi,\mathcal M_\alpha)\). The conformal action is transitive on
\(\mathcal M_{\alpha}/\mathbb R\), so there is a conformal map \(\tau_0\) such
that \(T_{\tau_0}\Theta_0\) is constant. By \eqref{eq:Paneitz-invariance}, after
subtracting its mean, \(T_{\tau_0}\Phi\) lies in the preceding neighbourhood of
zero. The local nearest-point construction gives a unique nearest extremal.
Transporting it back yields \(\Theta\).

The first variation of the squared distance at \(\Theta\) is zero in every
tangent direction to \(\mathcal M_{\alpha}/\mathbb R\). Choose \(\tau\) so that
\(T_\tau\Theta\) is constant and subtract that constant from \(T_\tau\Phi\).
Because the affine Jacobian terms cancel in a difference,
\[
T_\tau\Phi-T_\tau\Theta=(\Phi-\Theta)\circ\tau.
\]
The tangent space at the normalized extremal is \(\mathcal H_1\), by
\eqref{eq:extremal-derivative}. Hence the transformed remainder \(\Psi\) is
Paneitz-orthogonal to \(\mathcal H_1\). Since
\(P_{4}^{\mathbb S^{4}}Y=24Y\) for \(Y\in\mathcal H_1\), this is equivalent to
\(L^{2}\)-orthogonality. Subtracting the mean gives the
\(\mathcal H_0\)-condition in \eqref{eq:modulation-decomposition}.

The extremal \(\Theta\) realizes the infimum. The conformal invariance
\eqref{eq:Paneitz-invariance} therefore gives the exact identity
\eqref{eq:distance-exact}; \eqref{eq:deficit-invariance} gives
\eqref{eq:deficit-exact-modulation}. Finally, the Paneitz Poincar\'e inequality
on \(\mathcal H_0^{\perp}\) gives \eqref{eq:modulation-H2-control}. Uniformity on
compact \(\alpha\)-intervals follows from the uniform invertibility of the
matrix in \eqref{eq:normal-equations-derivative} and the smooth dependence of
\(m_\alpha\) on \(\alpha\).
\end{proof}

\begin{proof}[Proof of Theorem~\ref{thm:Adams-Choquard-stability}]
Let \(d_{P}(\Phi,\mathcal M_{\alpha})<\varepsilon_{\alpha}\). Apply Lemma
\ref{lem:local-modulation}. By \eqref{eq:deficit-exact-modulation},
\eqref{eq:distance-exact} and Proposition~\ref{prop:gauge-fixed},
\begin{equation*}
\mathcal D_{\alpha}(\Phi)
=\mathcal D_{\alpha}(\Psi)
\geq c_{\alpha}
\int_{\mathbb S^{4}}\Psi P_{4}^{\mathbb S^{4}}\Psi\,d\xi
=c_{\alpha}d_{P}(\Phi,\mathcal M_{\alpha})^{2}.
\end{equation*}
This proves \eqref{eq:local-stability-intro}.

We now prove the sharp asymptotic constant. Put
\begin{equation*}
\lambda_{k}:=k(k+1)(k+2)(k+3).
\end{equation*}
For \(k\geq2\), the quotient of the quadratic part of the deficit on
\(\mathcal H_k\) by the Paneitz energy is
\begin{equation}\label{eq:mode-quotient}
\frac{\Lambda_{k}(\alpha)}{2\lambda_{k}}
=\frac12\left[1-
\frac{\kappa_{\alpha}(\mu_{k}(\alpha)+\mu_{0}(\alpha))}{\lambda_{k}}
\right].
\end{equation}
The positive sequence
\(\kappa_{\alpha}(\mu_k(\alpha)+\mu_0(\alpha))\) is strictly decreasing,
whereas \(\lambda_k\) is strictly increasing. Hence their quotient is
strictly decreasing. It follows from \eqref{eq:mode-quotient} that the mode
quotient is strictly increasing for \(k\geq2\), and its minimum is attained
at \(k=2\).

To prove the lower limit, let \(\Phi_n\) be any sequence such that
\begin{equation*}
d_n:=d_P(\Phi_n,\mathcal M_{\alpha})\longrightarrow0,
\qquad d_n>0.
\end{equation*}
Let \(\Psi_n\) be the modulated remainders supplied by Lemma
\ref{lem:local-modulation}, and set \(Z_n:=\Psi_n/d_n\). Then
\begin{equation}\label{eq:normalized-remainders}
Z_n\perp_{L^{2}}\mathcal H_0\oplus\mathcal H_1,
\qquad
\int_{\mathbb S^{4}}Z_nP_{4}^{\mathbb S^{4}}Z_n\,d\xi=1,
\qquad
\|Z_n\|_{H^{2}}\leq C(\alpha).
\end{equation}
Lemma~\ref{lem:deficit-expansion}, \eqref{eq:distance-exact} and
\eqref{eq:modulation-H2-control} give
\begin{equation}\label{eq:normalized-deficit-limit}
\frac{\mathcal D_{\alpha}(\Phi_n)}{d_n^{2}}
=\frac12\langle\mathcal L_{\alpha}Z_n,Z_n\rangle+O_{\alpha}(d_n).
\end{equation}
Write \(Z_n=\sum_{k\geq2}Z_{n,k}\), with
\(Z_{n,k}\in\mathcal H_k\). By \eqref{eq:normalized-remainders},
\[
\sum_{k\geq2}\lambda_k\|Z_{n,k}\|_{2}^{2}=1.
\]
Therefore
\begin{equation}\label{eq:mode-lower-bound-sequence}
\frac12\langle\mathcal L_{\alpha}Z_n,Z_n\rangle
=\sum_{k\geq2}\frac{\Lambda_k(\alpha)}{2\lambda_k}
\lambda_k\|Z_{n,k}\|_2^{2}
\geq\frac{\Lambda_2(\alpha)}{240}.
\end{equation}
Equations \eqref{eq:normalized-deficit-limit} and
\eqref{eq:mode-lower-bound-sequence} prove the required liminf inequality.

For the reverse inequality, choose \(Y\in\mathcal H_2\) with
\(\int YP_{4}^{\mathbb S^{4}}Y=1\) and put \(\Phi_t=tY\). Differentiate the nearest-point equations
\eqref{eq:normal-equations-map} at \(t=0\). Their derivative in the
\(f\)-variable applied to \(Y\) is
\((Y,(8-\alpha)\xi_j)_P=0\) for every \(j\), because
\(Y\perp_P\mathcal H_1\). Since the matrix
\eqref{eq:normal-equations-derivative} is invertible, \(a'(0)=0\). The smooth implicit-function parameter therefore satisfies
\(a(t)=O(t^{2})\). Since
\[
\widehat m_{\alpha}(a)
=(8-\alpha)a\cdot\xi+O(|a|^{2})
\quad\hbox{in }H^{2}(\mathbb S^{4}),
\]
and \(Y\perp_{P}\mathcal H_{1}\), the cross term between \(tY\) and the
linear part of \(\widehat m_{\alpha}(a(t))\) vanishes. Therefore
\begin{equation*}
d_P(\Phi_t,\mathcal M_\alpha)^{2}=t^{2}+O(t^{4}).
\end{equation*}
whereas Lemma~\ref{lem:deficit-expansion} gives
\begin{equation*}
\mathcal D_\alpha(\Phi_t)
=\frac12\Lambda_2(\alpha)t^{2}\|Y\|_2^{2}+O(t^{3})
=\frac{\Lambda_2(\alpha)}{240}t^{2}+O(t^{3}).
\end{equation*}
This proves
\begin{equation*}
\lim_{\varepsilon\downarrow0}
\inf_{0<d_{P}(\Phi,\mathcal M_{\alpha})<\varepsilon}
\frac{\mathcal D_{\alpha}(\Phi)}
{d_{P}(\Phi,\mathcal M_{\alpha})^{2}}
=\frac{\Lambda_{2}(\alpha)}{240}.
\end{equation*}
Equivalently, for every \(\delta>0\) there exists
\(\varepsilon_{\alpha,\delta}>0\) such that
\begin{equation*}
\mathcal D_{\alpha}(\Phi)
\geq(\gamma_{\alpha}-\delta)
 d_P(\Phi,\mathcal M_{\alpha})^{2}
\end{equation*}
whenever \(d_P(\Phi,\mathcal M_\alpha)<\varepsilon_{\alpha,\delta}\).

By \eqref{eq:mu-monotone},
\begin{equation*}
\frac{\mu_{2}(\alpha)}{\mu_{0}(\alpha)}
=\frac{\alpha(\alpha+2)}{(8-\alpha)(10-\alpha)}.
\end{equation*}
Using \eqref{eq:kappa-identities},
\begin{equation*}
\Lambda_{2}(\alpha)
=120-\frac{6(\alpha^{2}-8\alpha+40)}{10-\alpha}.
\end{equation*}
Division by \(240\) gives the formula for \(\gamma_{\alpha}\) in
\eqref{eq:sharp-local-constant}. Finally,
\begin{equation*}
\gamma_{\alpha}-\frac25
=\frac{\alpha(4-\alpha)}{40(10-\alpha)}>0.
\end{equation*}
Uniformity for \(\alpha\) in compact subsets of \((0,4)\) follows from Lemmas
\ref{lem:deficit-expansion} and \ref{lem:local-modulation}, together with the
uniform spectral gap in Theorem~\ref{thm:full-spectrum}.
\end{proof}

\subsection{Euclidean formulation}

Define the conformal energy space
\begin{equation}\label{eq:Euclidean-X}
\mathcal X
:=\left\{\phi=S^{*}\Phi:\Phi\in H^{2}(\mathbb S^{4})\right\}.
\end{equation}
Equivalently, \(\phi\in\mathcal X\) if and only if
\(S_{*}\phi\in H^{2}(\mathbb S^{4})\). For \(\phi\in\mathcal X\), put
\begin{equation*}
\mathcal J_{\alpha}(\phi)
:=\int_{\mathbb R^{4}}\int_{\mathbb R^{4}}
\frac{e^{U(x)+\phi(x)}e^{U(y)+\phi(y)}}
{|x-y|^{\alpha}}\,dxdy.
\end{equation*}

The sphere-to-Euclidean passage on the form domain is summarized next.

\begin{lemma}\label{lem:energy-space-transfer}
The map \(S^{*}:H^{2}(\mathbb S^{4})\to\mathcal X\) is one-to-one and onto.
If \(\phi=S^{*}\Phi\), then
\begin{equation}\label{eq:energy-space-transfer}
\int_{\mathbb S^{4}}|\Phi|^{2}\,d\xi
=\int_{\mathbb R^{4}}\phi(x)^{2}\omega(x)^{4}\,dx,
\qquad
\int_{\mathbb S^{4}}\Phi P_{4}^{\mathbb S^{4}}\Phi\,d\xi
=\int_{\mathbb R^{4}}|\Delta\phi|^{2}\,dx.
\end{equation}
In particular, every \(\phi\in\mathcal X\) belongs to
\(H^{2}_{\operatorname{loc}}(\mathbb R^{4})\), satisfies
\(\Delta\phi\in L^{2}(\mathbb R^{4})\), and has finite weighted
\(L^{2}\)-norm.
\end{lemma}

\begin{proof}
The first identity is the change of variables in
\eqref{eq:stereo-identities}. For smooth \(\Phi\), the second identity is the
conformal energy identity used in the proof of Proposition
\ref{prop:conformal-linearized}. Let now \(\Phi\in H^{2}(\mathbb S^{4})\),
and choose \(\Phi_n\in C^{\infty}(\mathbb S^{4})\) converging to \(\Phi\) in
\(H^{2}\). The spectral description \eqref{eq:Paneitz-eigen} shows that
\[
\|\Phi\|_{L^{2}}^{2}
+\int_{\mathbb S^{4}}\Phi P_{4}^{\mathbb S^{4}}\Phi\,d\xi
\]
is an equivalent squared norm on \(H^{2}(\mathbb S^{4})\). Consequently,
\(S^{*}\Phi_n\) is Cauchy both in the weighted \(L^{2}\)-norm and in the
\(L^{2}\)-norm of its Laplacian. Its limit is \(S^{*}\Phi\) locally, and
passing to the limit proves \eqref{eq:energy-space-transfer}. Local Sobolev
regularity follows from the smoothness of the stereographic coordinate map
away from the pole. Injectivity and surjectivity follow directly from the
definition of \(\mathcal X\).
\end{proof}

\begin{remark}\label{rem:Euclidean-space-necessity}
The pullback condition in \eqref{eq:Euclidean-X} is essential. The three
requirements \(\phi\in H^{2}_{\operatorname{loc}}\),
\(\Delta\phi\in L^{2}\), and
\(\int\phi^{2}\omega^{4}<\infty\) do not by themselves imply that
\(S_{*}\phi\in H^{2}(\mathbb S^{4})\). For example, a nonconstant affine
harmonic function satisfies those three requirements but has a nonremovable
singularity at the stereographic pole. The space \(\mathcal X\) excludes such
harmonic remainders and is exactly the Euclidean realization of the spherical
Paneitz energy space.
\end{remark}

The spherical inequality and local stability theorem therefore have the following Euclidean form.

\begin{corollary}\label{cor:Euclidean-stability}
For every \(\phi\in\mathcal X\),
\begin{equation}\label{eq:Euclidean-Adams-Choquard}
\begin{aligned}
\log\frac{\mathcal J_{\alpha}(\phi)}{8\pi^{2}(8-\alpha)}
\leq{}&
\frac{3}{4\pi^{2}}\int_{\mathbb R^{4}}\phi(x)\omega(x)^{4}\,dx+\frac{1}{8\pi^{2}(8-\alpha)}
\int_{\mathbb R^{4}}|\Delta\phi|^{2}\,dx.
\end{aligned}
\end{equation}
Equality holds if and only if
\begin{equation}\label{eq:Euclidean-equality}
\phi(x)=c+U_{\mu,\zeta}(x)-U(x)
\end{equation}
for some \(c\in\mathbb R\), \(\mu>0\), and \(\zeta\in\mathbb R^{4}\).

Let \(\mathcal D_{\alpha}^{\mathbb R^{4}}(\phi)\) denote the right-hand side of \eqref{eq:Euclidean-Adams-Choquard} minus its left-hand side, multiplied by \(4\pi^{2}(8-\alpha)\). There exist \(\varepsilon_{\alpha}>0\) and \(c_{\alpha}>0\) such that
\begin{equation}\label{eq:Euclidean-stability}
\mathcal D_{\alpha}^{\mathbb R^{4}}(\phi)
\geq c_{\alpha}
\inf_{\substack{\mu>0\\ \zeta\in\mathbb R^{4}}}
\int_{\mathbb R^{4}}
\left|\Delta\bigl(\phi-U_{\mu,\zeta}+U\bigr)\right|^{2}\,dx
\end{equation}
whenever the infimum on the right-hand side is smaller than \(\varepsilon_{\alpha}^{2}\). The sharp asymptotic constant relative to this squared distance is \(\gamma_{\alpha}\) from \eqref{eq:sharp-local-constant}.
\end{corollary}

\begin{proof}
Let \(\Phi=S_{*}\phi\). By Lemma~\ref{lem:energy-space-transfer},
\(\Phi\in H^{2}(\mathbb S^{4})\) and
\begin{equation*}
\int_{\mathbb S^{4}}\Phi P_{4}^{\mathbb S^{4}}\Phi\,d\xi
=\int_{\mathbb R^{4}}|\Delta\phi|^{2}\,dx.
\end{equation*}
Moreover,
\begin{equation*}
\int_{\mathbb S^{4}}\Phi\,d\xi
=\int_{\mathbb R^{4}}\phi(x)\omega(x)^{4}\,dx.
\end{equation*}
The spherical Adams inequality implies \(e^{\Phi}\in L^{q_{\alpha}}(\mathbb S^{4})\), so the Hardy--Littlewood--Sobolev inequality shows that all nonlocal integrals below are finite. Using \eqref{eq:stereo-identities} and the formula for \(e^{U}\) in the proof of Proposition~\ref{prop:conformal-linearized},
\begin{equation*}
\mathcal J_{\alpha}(\phi)
=\kappa_{\alpha}\mathcal I_{\alpha}(\Phi).
\end{equation*}
By \eqref{eq:kappa-identities} and \eqref{eq:sphere-volume},
\begin{equation*}
\mathcal J_{\alpha}(0)
=\kappa_{\alpha}|\mathbb S^{4}|\mu_{0}(\alpha)
=8\pi^{2}(8-\alpha).
\end{equation*}
Thus \eqref{eq:Euclidean-Adams-Choquard} is exactly \eqref{eq:sharp-Adams-Choquard}. The equality set is the stereographic image of \(\mathcal M_{\alpha}\), which is \eqref{eq:Euclidean-equality}. Finally, the Paneitz distance becomes the homogeneous biharmonic distance in \eqref{eq:Euclidean-stability}. Theorem~\ref{thm:Adams-Choquard-stability} proves the remaining assertions.
\end{proof}

\noindent{\bf Acknowledgements}

The authors were supported by the National Natural Science Foundation of China (No. 12371121) and the Fundamental Research Funds for the Central Universities (No. SWU-KF26002).

\noindent{\bf Data availability statement}

Data availability is not applicable to this article as no datasets were generated or analysed during the current study.

\noindent{\bf Conflict of interest}

The authors declare that there is no conflict of interest.

\end{document}